\documentclass[11pt]{article}
\usepackage{graphicx}
\usepackage{amsmath,amssymb,amsthm,amsfonts}
\usepackage{color}
\usepackage{amssymb}
\usepackage{cite}
\newtheorem{thm}{Theorem}[section]

\newtheorem{lem}[thm]{Lemma}
\newtheorem{prop}[thm]{Proposition}

\theoremstyle{definition}

\theoremstyle{remark}
\usepackage{appendix}

\numberwithin{equation}{section}

\DeclareMathSymbol{\C}{\mathalpha}{AMSb}{"43}

\newcommand{\R}{{\mathbb{R}}}

\newcommand{\inte}{\int_{\mathbb{R}}}

\def\R{{\mathbb R}}
\def\C{{\mathbb C}}
\allowdisplaybreaks[4]

\newcommand{\bsub}{\begin{subequations}}
\newcommand{\esub}{\end{subequations}$\!$}

\begin{document}

\title{Ground States of One-Dimensional Fermionic  Schr\"{o}dinger Systems Near a Critical Exponent}

\author{
Bin Chen\thanks{Academy of Mathematics and Systems Science, Chinese Academy of Sciences, Beijing 100190, P. R. China. B. Chen  is partially supported by NSF of China (Grant 12501151).
Email:  binchen@amss.ac.cn.},
Yujin Guo\thanks{School of Mathematics and Statistics,  Key Laboratory of Nonlinear Analysis $\&$ Applications (Ministry of Education), Central China Normal University, Wuhan 430079, P. R. China. Y. J. Guo is partially
supported by NSF of China (Grants 12225106 and
12371113) and  National Key R \& D Program of China (Grant 2023YFA1010001). Email: yguo@ccnu.edu.cn.},
\ Yong Luo\thanks{School of Mathematics and Statistics, Hubei Key Laboratory of Mathematical Sciences, Central China Normal University, Wuhan 430079, P. R. China. Y. Luo is partially supported by NSF of China (Grant 12571119). Email:  yluo@ccnu.edu.cn.},
and Juncheng Wei\thanks{Department of Mathematics, Chinese University of Hong Kong, Shatin, NT, Hong Kong.  J. C. Wei is partially supported by GRF from RGC of Hong Kong entitled ``New frontiers in singularity formations in nonlinear partial differential equations". Email: wei@math.cuhk.edu.hk.}
}


\date{\today}

\smallbreak \maketitle
\begin{abstract}
We study ground states of the   fermionic nonlinear Schr\"{o}dinger system $J_2(p)$ in $\R$, where $p>1$ denotes a polynomial exponent of the nonlinear term. It is known that the system $J_2(p)$ admits ground states for any $1<p<2$, while there is no ground state for $J_2(2)$. We prove that  there is no ground state of $J_2(p)$ as $p\searrow 2$, which addresses the special case of Conjecture 5 in [D. Gontier, M. Lewin and F. Q. Nazar, ARMA, 2021]. The refined limiting profile of ground states for $J_2(p)$ is also analyzed as $p\nearrow 2$, which shows that the corresponding density admits exactly two bumps whose distance goes up to infinity as  $p\nearrow 2$.
\end{abstract}

\vskip 0.05truein



\vskip 0.2truein


\section{Introduction}

In this paper, we consider the following variational problem of fermionic nonlinear Schr\"{o}dinger (NLS) systems in $\R$:
\begin{equation}\label{1.1}
\begin{split}
J_2(p):=&\inf\Big\{\mathcal{E}_p(u_{1}, u_{2}):\,   u_i\in H^1(\R), \ \langle u_i,u_j\rangle_{L^2}=\delta_{ij},\  i,j=1,2\Big\},\,  \ 1<p<3,
\end{split}
\end{equation}
where the energy functional $\mathcal{E}_p(u_{1} ,u_{2})$ satisfies
\begin{equation}\label{1.2}
\mathcal{E}_p(u_{1} ,u_{2}):=\inte \big(|\nabla u_{1}|^2 +|\nabla u_{2}|^2\big)dx -\frac{1}{p}\int_{\R}(u_{1}^2+u_{2}^2  )^pdx.
\end{equation}
The problem $J_2(p)$ is often used to describe ground states of many-body  quantum fermions, such as electrons, protons, neutrons, and so on. We also refer \cite{G} and the references therein to the physical experiments of many-body  quantum fermionic systems in one dimension. Since the analysis of $J_2(p)$ presents many novel difficulties, together with its strong physical motivations, the problem $J_2(p)$ has attracted a lot of attentions over the past few years, see \cite{i,ii,Lewin,1981} and the references therein. We also comment that the problem $J_2(p)$ for the case $p\ge 3$ was discussed recently in \cite{i,ii,Lewin}. Especially, one can obtain the nonexistence of minimizers for $J_2(p)$ for any $p> 3$.

Suppose $(u_{1}, u_{2})$ is a minimizer of $J_2(p)$,  then up to an orthogonal transformation, the variational theory gives that $(u_{1}, u_{2})$ satisfies the
following fermionic NLS system
\begin{eqnarray}\label{1:x4}
	\left\{
	\begin{array}{lll}
		\!\!\!-u''_{1}-\big(u_{1}^2+u_{2}^2\big)^{p-1}\, u_{1}=\mu_{1}u_{1}  & \mathrm{in} \,\ \R,\\[1mm]
		\!\!\!-u''_{2}-\big(u_{1}^2+u_{2}^2\big)^{p-1}\, u_{2}=\mu_{2}u_{2}  & \mathrm{in} \,\ \R,
	\end{array}
	\right.
\end{eqnarray}
where $1<p<3$ and $(\mu_1, \mu_2)\in \R\times\R$ satisfying $\mu_1\leq\mu_2$ is a suitable Lagrange multiplier. It deserves to mention that even though the system (\ref{1:x4}) was studied widely over the past few decades, see \cite{Wei,K} and the references therein, the constraint condition $\langle u_i,u_j\rangle_{L^2}=\delta_{ij}$ was usually not considered in these mentioned works. We also observe that the complete classification on the solutions of the system (\ref{1:x4}) with $p=2$ was proved successfully in \cite{ii} by investigating the corresponding integral system of (\ref{1:x4}).
It follows from \cite{i,ii} that if $(u_{1}, u_{2})$ is a minimizer of $J_2(p)$, then its associated Lagrange multiplier $(\mu_1, \mu_2) $ given in (\ref{1:x4})  satisfies $\mu_{1}\le\mu_{2}<0$, which are the $2$-first eigenvalues of the operator $-D_{xx}-\big(u_{1}^2+u_{2}^2\big)^{p-1}$ in $\R$.
Then we further obtain from \cite[Theorems 11.6 $\&$ 11.8]{analysis} that $\mu _1$ is simple, and its associated eigenfunction $u_1$ can be chosen to be strictly positive. Thus, if $(u_{1} ,u_{2})$ is a minimizer of $J_2(p)$,   throughout the present paper, then we always assume that $(u_{1}, u_{2})$ satisfies (\ref{1:x4}) for some $\mu_{1}<\mu_{2}<0$, and $u_{1}(x)>0$ holds in $\R$. Hence, any minimizer $(u_{1} ,u_{2})$ of $J_2(p)$ can be also called a {\em ground state} of $J_2(p)$, which is consist with \cite[Definition 1]{i}.


There are some existing analytical results of the problem $J_2(p)$. More precisely, it was proved in \cite[Theorem 4]{i} that $J_2(p)$ admits minimizers for any $1<p<2$ by an adaptation of the classical concentration compactness principle \cite{concen,concen2}. Employing Hirota's bilinearisation method, it was further proved in \cite[Theorem 8]{ii} that there is no minimizer of $J_2(p)$ at $p=2$. It turns out from \cite{i,ii} that the orthonormal constraint of $J_2(p)$ leads to the essential  difficulties in the analysis of $J_2(p)$. Moreover, since $J_2(p)$  has both the translation invariance and the rotation invariance, the extra difficulties appear in investigating the existence and other analytical properties of its minimizers. On the other hand, the authors of \cite{i} posed a conjecture that $J_2(p)$ does not have any minimizer for all $2\le p<3$, see \cite[Conjecture 5]{i} for more details. These illustrate that $p=2$ behaves like a critical exponent of $J_2(p)$ in one dimension. Furthermore, some numerical computations of minimizers for $J_2(p)$ presented in \cite[Theorem 8]{ii} interesting new phenomena as $p \nearrow 2$. Stimulated by these facts, the main purpose of the present paper is to address the nonexistence of minimizers for $J_2(p)$ as $p\searrow 2$, and analyze as well the refined limiting profile of minimizers for $J_2(p)$ as $p \nearrow 2$.

In order to introduce the main results of the present paper, we first recall from \cite{uniq} that up to translation, the following scalar field equation
\begin{equation}\label{1.3}
	u''-u+u^3= 0\ \   \mathrm{in} \, \  \R
\end{equation}
admits a unique positive solution $w>0$, which satisfies
\begin{equation}\label{a8}
	w(x)=\frac{2\sqrt 2}{e^{x}+e^{-x}},\ \    a^*:=\|w\|^2_2=4,
\end{equation}
and admits the nondegeneracy. By analyzing the refined limiting profiles of minimizers for $J_2(p)$ as $p\searrow 2$, the first main result of the present paper is concerned with the following nonexistence of minimizers for $J_2(p)$.

\begin{thm}\label{thm1.1}
There exists a sufficiently small $\delta >0$ such that there is no minimizer of $J_2(p)$ for any $2\le p\le 2+\delta$.
\end{thm}

The nonexistence of Theorem \ref{thm1.1} addresses a special case of \cite[Conjecture 5]{i}. Since it follows from \cite[Theorem 8]{ii} that $J_2(2)$ does not admit any minimizer, by contradiction we shall prove
Theorem \ref{thm1.1} for the case $2< p\le 2+\delta$. On the contrary, in the following we always assume that $J_2(p_n)$ admits a minimizer $(u_{1n}, u_{2n})$ as $n\to\infty$, where $p_n \searrow 2$ as $n\to\infty$ and $(u_{1n}, u_{2n})$ satisfies the elliptic system \eqref{1:x4} with a suitable Lagrange multiplier $(\mu_{1n}, \mu_{2n})$. Without loss of generality, we may further assume that
\begin{equation}\label{1.6K}
u_{1n}(0)=\max\limits_{x\in\R}  u_{1n}(x)>0,\  \  u_{2n}(x_n)=\max\limits_{x\in\R} u_{2n}(x),   \ \ x_n>0.
\end{equation}
We next briefly sketch the proof strategy of Theorem \ref{thm1.1} by the following three steps:

As the first step of proving Theorem \ref{thm1.1}, we shall prove that $ \lim\limits_{n\to\infty}x_n=\infty$ and 
\begin{equation}\label{1.6}
\begin{split}
(u_{1n},  u_{2n})
\to \big(\sqrt{1-a_0^2}w_{*}(x)+a_0w_{*}(x- x_n),\ -a_0w_{*}(x)+\sqrt{1-a_0^2}w_{*}(x-x_n)\big)
\end{split}
\end{equation}
strongly in $L^{2}(\R)\cap L^{\infty}(\R)$ as $n\to\infty$ for some  constant $a_0\in[0, \sqrt{2}/2]$, where $w_*(x)\equiv\frac{1}{4}w(\frac{1}{4}x)>0$ in $\R$.  Because  the energy functional $\mathcal{E}_{p_n}(u_{1n} ,u_{2n})$ is invariant under orthogonal transformations, unfortunately  we are here unable to determine the specific value of  $a_0$. This is different from the existing convergence results (cf. \cite{GLY,glw, GS} and the references therein), which were mainly concerned with the limiting behavior of bosonic systems.


Based on \eqref{1.6}, the second step of proving Theorem \ref{thm1.1} is to derive the expansion of  $(u_{1n}, u_{2n})$ as $n\to \infty$. Employing the orthonormal condition  \eqref{1.1} and the non-degeneracy of $w$, towards this aim, we first choose technically in Section 3 a sequence $\{(a_n, b_n, \delta_n, \eta_n)\}$ satisfying  $(a_n, b_n, \delta_n, \eta_n)=[1+o(1)](a_0, -a_0, 0, 0)$ as $n\to\infty$ such that
\begin{eqnarray}\label{1.7}
	\left\{
	\begin{array}{lll}
		\!	\!	\!	 u_{1n}(x)=\sqrt{1-b^2_n}\widetilde w_{1n}(x-\delta_n)+a_n\widetilde w_{2n}(x- x_n+\eta_n)+ \phi_n(x), \\[1mm]
		\!	\!	\!	 u_{2n}(x)=b_n\widetilde w_{1n}(x-\delta_n)+\sqrt{1-a_n^2}\widetilde w_{2n}(x- x_n+\eta_n)+ \psi_n(x),
	\end{array}
	\right.
\end{eqnarray}
and  the error terms $\phi_n(x)$ and $\psi_n(x)$ satisfy $\|\phi_n\|_\infty+\|\psi_n\|_\infty=O\big(e^{-\sqrt{|\mu_{2n}|}x_n}\big)\ \mbox{as}\ n\to\infty$,
where  the function $\widetilde w_{in}$ satisfying (\ref{AA3.1}) is closely related to the $i$-th eigenvalue of the operator  $-D_{xx}-\big(u_{1n}^2+u_{2n}^2\big)^{p_n-1}$ in $\R$.  Furthermore, to determine the specific value of the constant $a_0$,   we  introduce the following new transformation:
\begin{eqnarray*}
	\left\{
	\begin{array}{lll}
		\!	\!	\!	 \hat u_{1n}(x):=\sqrt{1-b^2_n} u_{1n}(x)+b_n  u_{2n}(x), \\[1mm]
		\!	\!	\!	 \hat u_{2n}(x):=a_n u_{1n}(x)+\sqrt{1-a_n^2} u_{2n}(x).
	\end{array}
	\right.
\end{eqnarray*}
By analyzing the system of $(\hat u_{1n}, \hat u_{2n})$, instead of $(u_{1n}, u_{2n})$,  we shall finally prove that $\lim_{n\to\infty}a_n=a_0=\sqrt{2}/2$.

The third step of proving Theorem \ref{thm1.1} is to derive a contradiction by establishing the refined estimate of the global maximal point $x_n$ as $n\to \infty$. More precisely, by deriving delicately several Pohozaev-type identities, we shall conclude from (\ref{1.7}) that $(2-p_n)x_n^2=48+o(1)$ as $n\to\infty$. It thus implies that $p_n>2$ cannot occur, a contradiction. This therefore completes the proof of Theorem \ref{thm1.1}. We refer Sections 2-4 for the detailed proof of Theorem \ref{thm1.1}.

One can observe from Sections 2-4 that the above proof strategy of Theorem \ref{thm1.1} holds essentially for the limiting case $p\to 2$. As a byproduct of proving Theorem \ref{thm1.1}, we are thus able to address the refined limiting profile of minimizers for $J_2(p)$ as $p\nearrow 2$. Towards this result, we define for $i=1,2,$
\begin{equation}\label{1:win}
\widetilde w_{in}(x):=\Big(\frac{2\sqrt{|\mu_{in}|p_n}}{e^{(p_n-1)\sqrt{|\mu_{in}|}x}
+e^{(1-p_n)\sqrt{|\mu_{in}|}x}}\Big)^{\frac{1}{p_n-1}}>0,
\end{equation}
where $\mu_{1n}<\mu_{2n}<0$ are the same as in (\ref{1:x4}). In Section 4 we shall establish the following limiting profile of minimizers for $J_2(p)$ as $p\nearrow 2$.

\begin{prop}\label{prop1.2}
Let $(u_{1n}, u_{2n})$ be a minimizer of $J_2(p_n)$ satisfying \eqref{1:x4}, where  $p_n\nearrow 2$ as $n\to\infty$,  $u_{1n}(0)=\max_{x\in\R}u_{1n}(x)$ and $u_{2n}(0)\leq0$. Then there exists a subsequence, still denoted by $\big\{(u_{1n}, u_{2n})\big\}$,  of $\big\{(u_{1n}, u_{2n})\big\}$   such that up to   reflections (across the y-axis) if necessary,
\begin{eqnarray}\label{prop1.1}
\left\{
\begin{array}{lll}
\!	\!	\!	 u_{1n}(x)=\frac{\sqrt{2}}{2} \widetilde w_{1n}(x)+\frac{\sqrt{2}}{2}\widetilde w_{2n}(x-x_n)+O(x_ne^{-\frac{x_n}{4}})\ \ as\,\ n\to\infty, \\[1mm]
\!	\!	\!	 u_{2n}(x)=-\frac{\sqrt{2}}{2} \widetilde w_{1n}(x)+\frac{\sqrt{2}}{2} \widetilde w_{2n}(x-x_n)+O(x_ne^{-\frac{x_n}{4}})\ \ as\,\ n\to\infty,
\end{array}
\right.
\end{eqnarray}
where  $\widetilde w_{in}(x)$ is given by (\ref{1:win}), 
and the global maximal point $x_n>0$ of $u_{2n}$ satisfies  $x_n=[4\sqrt{3}+o(1)](2-p_n)^{-\frac{1}{2}}$ as $n\to\infty$.
\end{prop}

It follows from Proposition \ref{prop1.2} that the density $\rho_n(x)$ of minimizers for $J_2(p_n)$ satisfies
\begin{equation}\label{1:winM}
\rho_n(x):=u^2_{1n}(x)+u^2_{2n}(x)=\widetilde w^2_{1n}(x)+\widetilde w^2_{2n}(x-x_n)+O(x_ne^{-\frac{x_n}{4}})\ \ \mbox{as}\,\ p_n\nearrow 2,
\end{equation}
and hence the density $\rho_n(x)$ admits exactly two bumps, as expected in \cite[Section 1.4]{i}, whose distance $x_n$ goes up to infinity as $p_n\nearrow 2$. We predict that Proposition \ref{prop1.2} can further yield other analytical properties of minimizers for $J_2(p)$ as $p\nearrow 2$, which are investigated in the companion work of the present paper.

This paper is organized as follows. Section 2 is devoted to the proof of Proposition \ref{J1} on the convergence of minimizers for $J_2(p)$ as $p\to 2$. In Section 3, by establishing Lemma \ref{a} we derive the limiting behavior (\ref{1.7}) as  $p\to 2$. We first analyze in Section 4 the refined estimate of the global maximal point $x_n>0$ defined in (\ref{1.6K}), based on which we then prove  Theorem  \ref{thm1.1} and Proposition \ref{prop1.2} in Subsection 4.1.  Appendix A is finally focussed on the proofs of Lemmas \ref{lemA.1} and \ref{lemA.2} used in Section 4.

\section{Convergence of Minimizers as $p\to 2$}

Starting from this section, without specific notations, we always suppose that $J_2(p)$ has minimizers as $p\to 2$.
In this section, we mainly address the convergence of minimizers for $J_2(p)$ as $p\to 2$. Following (\ref{1:x4}), let  $(u_{1n} ,u_{2n})$ be a minimizer of $J_2(p_n)$, where $p_n\to2$ as $n\to\infty$, and satisfy the following fermionic NLS system
\begin{eqnarray}\label{x4}
	\left\{
	\begin{array}{lll}
		\!\!\!-u''_{1n}-\big(u_{1n}^2+u_{2n}^2\big)^{p_n-1}\, u_{1n}=\mu_{1n}u_{1n}  & \mathrm{in} \,\ \R,\\[1mm]
		\!\!\!-u''_{2n}-\big(u_{1n}^2+u_{2n}^2\big)^{p_n-1}\, u_{2n}=\mu_{2n}u_{2n}  & \mathrm{in} \,\ \R,
	\end{array}
	\right.
\end{eqnarray}
where $\mu_{1n}<\mu_{2n}<0$ are the $2$-first eigenvalues of the operator $-D_{xx}-\big(u_{1n}^2+u_{2n}^2\big)^{p_n-1}$ in $\R$. Applying (\ref{a8}), in this paper we always denote
 \begin{equation}\label{2.8b}
 	w_\ast(x):
 = \frac{\sqrt{2}}{2}
 \frac{1}{e^{\frac{x}{4}} + e^{-\frac{x}{4}}}>0
\end{equation}
to be the unique (up to translations)  positive solution of the following equation
\begin{equation}\label{2:limitE}
- u'' + \frac{1}{16} u - u^3 = 0\ \ \ \text{in}\ \, \R.
\end{equation}
The main results of this section can be then stated as the following proposition.

\begin{prop}\label{J1}
Let $(u_{1n},u_{2n})$ be a minimizer of $J_2(p_n)$ satisfying \eqref{x4}, where $p_n\in[3/2, 5/2]$ satisfies $p_n\to2$ as $n\to\infty$.
Then there exist a constant $a_0\in[0, \sqrt{2}/2]$ and  a sequence  $\{\bar x_n\}\subset\R$ satisfying $\bar x_n\to\infty$ as $n\to\infty$ such that  up to translations and  reflections (across the y-axis) if necessary,
\begin{equation}\label{J2}
\begin{split}
&(u_{1n}, \pm u_{2n})\\
\to &\big(\sqrt{1-a_0^2}w_{*}(x)+a_0w_{*}(x-\bar x_n),\ -a_0w_{*}(x)+\sqrt{1-a_0^2}w_{*}(x-\bar x_n)\big)
\end{split}
\end{equation}
strongly in $L^{r}(\R)$ as $n\to\infty$, where $r\in[2, \infty]$ is arbitrary, $w_{*}(x)>0$ is as in (\ref{2.8b}), and the multiplier $\mu_{in}$ given by  \eqref{x4} satisfies
\begin{equation}\label{mun}
\lim\limits_{n\to\infty}\mu_{1n}=\lim\limits_{n\to\infty}\mu_{2n}=-1/16.
\end{equation}
\end{prop}

Since the problem $J_2(p_n)$ is translation invariant, without loss of generality, we now suppose  that
\begin{eqnarray}\label{J10}
	u_{1n}^2(0) + u_{2n}^2(0)
	= \max_{x\in\mathbb{R}} \bigl( u_{1n}^2(x) + u_{2n}^2(x) \bigr).
\end{eqnarray}
In order to prove Proposition \ref{J1}, we first establish the following lemma.

\begin{lem}\label{lem2.1}
Let $(u_{1n},u_{2n})$ be a minimizer of $J_2(p_n)$ satisfying \eqref{x4} and \eqref{J10}, where $p_n\in[3/2, 5/2]$ satisfies $p_n\to2$ as $n\to\infty$. Then we have the following conclusions:
\begin{enumerate}
\item [(i).]
There exists a constant $\delta>0$, independent of $n>0$, such that
\begin{equation}\label{J6}
u_{1n}^2(0) + u_{2n}^2(0) \ge \delta>0 \ \  \text{as}\ \, n\to\infty.
\end{equation}
	
\item [(ii).]  $(u_{1n},u_{2n})$ satisfies
\begin{equation*}
(u_{1n},u_{2n}) \rightharpoonup (u_1,u_2) \quad \text{weakly in } H^1(\mathbb{R})\ \, \text{as}\ \, n\to\infty,
\end{equation*}
and $(\mu_{1n}, \mu_{2n})$ given by \eqref{x4} satisfies
\begin{equation}\label{mu}
\lim\limits_{n\to\infty}\mu_{1n}=\lim\limits_{n\to\infty}\mu_{2n}=-1/16,
\end{equation}
where $(u_1,u_2)$ is given by
\begin{equation}\label{J9}
(u_1,u_2)
= \bigl( \sqrt{1-a^2}\, w_\ast,\; a\, w_\ast \bigr) \ \ for\ some \,\ a \in [-1,1],
\end{equation}
and $w_\ast(x)>0$ is as in (\ref{2.8b}).
\end{enumerate}
\end{lem}


\noindent{\bf Proof.}
$(i)$. We first prove that
\begin{equation}\label{2.7}
\lim\limits_{n\to\infty}J_2(p_n)=J_2(2)<0.
\end{equation}
Actually,  recall from  \cite{ii}  the following Lieb-Thirring inequality: for any orthonormal system $(\varphi_1, \varphi_2)$,
\begin{eqnarray*}\label{lt}
\inte\big(|\nabla \varphi_1|^2+|\nabla \varphi_2|^2\big)dx\geq c_{LT}\inte(\varphi_1^2+\varphi_2^2)^3dx,
\end{eqnarray*}
where the constant $c_{LT}>0$ is independent of $\varphi_i\in H^1(\R)$. Employing Young's inequality, we then deduce from \cite[Lemma 12 (i)]{i}  that
\begin{equation}\label{2.6}
\begin{split}
0\geq&\lim\limits_{n\to\infty}J_2(p_n)=\lim\limits_{n\to\infty}\mathcal{E}_{p_n}(u_{1n}, u_{2n})\\
\geq&\lim\limits_{n\to\infty}\Big[\inte\sum_{i=1}|\nabla u_{in}|^2dx-\frac{\varepsilon}{p_n}\inte\big(\sum_{i=1}u_{in}^2\big)^3dx-\frac{C(\varepsilon)}{p_n}\inte\big(\sum_{i=1}u_{in}^2\big)dx\Big]\\
\geq&\Big(1-\frac{\varepsilon}{ c_{LT}}\Big)\lim\limits_{n\to\infty}\inte\sum_{i=1}|\nabla u_{in}|^2dx-2C(\varepsilon),
\end{split}
\end{equation}
where $\varepsilon\in(0, c_{LT}/2)$, and $C(\varepsilon)>0$ is independent of $n>0$.
This yields that the sequence
\begin{equation}\label{2.8a}
\{(u_{1n}, u_{2n})\}\ \ \text{is  bounded uniformly in}\ H^{1}(\R)\times H^{1}(\R).
\end{equation}
Applying Taylor's expansion, we thus conclude that
\begin{equation}\label{2.8}
\begin{split}
&J_2(p_n)-J_2(2)
\geq\mathcal{E}_{p_n}(u_{1n}, u_{2n})-\mathcal{E}_{2}(u_{1n}, u_{2n})\\
=&\big(\frac{1}{2}-\frac{1}{p_n}\big)\inte(u_{1n}^2+ u_{2n}^2)^{p_n}dx-\frac{1}{2}\inte(u_{1n}^2+ u_{2n}^2)^{p_n}-(u_{1n}^2+ u_{2n}^2)^2dx\\
=&O(|p_n-2|)\ \ \ \text{as}\ \ n\to\infty.
\end{split}
\end{equation}
Here and below $C_n:=O(|p_n-2|)$ means that  there exists a constant $C>0$, independent of $n>0$, such that $|C_n|\leq C |p_n-2|$ holds for any sufficiently large $n>0$. Moreover, for any minimizing sequence $\{(v_{1n}, v_{2n})\}$ of $J_2(2)$, the same argument of \eqref{2.8} gives that
\begin{equation}\label{2.9}
\begin{split}
J_2(p_n)-J_2(2)\leq&\mathcal{E}_{p_n}(v_{1n}, v_{2n})-\mathcal{E}_{2}(v_{1n}, v_{2n})+o(1)\\
= &O(|p_n-2|)+o(1)\ \ \ \text{as}\ \ n\to\infty.
\end{split}
\end{equation}
By \cite[Lemma 12 (i)]{i}, we thus obtain from \eqref{2.8} and \eqref{2.9} that \eqref{2.7} holds true.

We now prove that  \eqref{J6} holds true. On the contrary, suppose \eqref{J6} is false. This then implies from (\ref{J10}) that
\begin{equation*}
u_{1n}^2(0)+u_{2n}^2(0)
= \max_{x\in\mathbb{R}} \bigl( u_{1n}^2(x)+u_{2n}^2(x) \bigr)
\rightarrow 0 \quad\hbox{as}\,\ n\to\infty,
\end{equation*}
and hence
\begin{equation*}
\int_{\mathbb{R}} (u_{1n}^2+u_{2n}^2)^{p_n}\,dx
\le \bigl( u_{1n}^2(0)+u_{2n}^2(0) \bigr)^{p_n-1}
\int_{\mathbb{R}} (u_{1n}^2+u_{2n}^2)\,dx
\rightarrow 0 \quad\hbox{as}\,\ n\to\infty.
\end{equation*}
We thus deduce that
\begin{equation*}
\lim_{n\to\infty}J_2(p_n)
= \lim_{n\to\infty}
\left\{
\int_{\mathbb{R}} \bigl( |u_{1n}'|^2+|u_{2n}'|^2 \bigr)\,dx
- \frac{1}{p_n} \int_{\mathbb{R}} (u_{1n}^2+u_{2n}^2)^{p_n}\,dx
\right\}
\ge 0,
\end{equation*}
which however contradicts with \eqref{2.7}. Thus, \eqref{J6} is proved, which therefore completes the proof of Lemma \ref{lem2.1} (i).

$(ii)$.
Since the proof of Lemma \ref{lem2.1} (i) gives that $\{(u_{1n},u_{2n})\}$ is bounded uniformly in $H^1(\mathbb{R})\times H^1(\mathbb{R})$,  there exist a subsequence, still denoted by $\{(u_{1n},u_{2n})\}$, of $\{(u_{1n},u_{2n})\}$ and a function $(u_1, u_2)\in H^1(\mathbb{R})\times H^1(\mathbb{R})$ such that $(u_{1n},u_{2n})$ satisfies
\begin{equation}\label{2.12}
(u_{1n},u_{2n}) \rightharpoonup (u_1,u_2)
\ \  \text{weakly in } H^1(\mathbb{R})\times H^1(\mathbb{R}) \  \ \text{as}\ \ n\to\infty,
\end{equation}
and $(\mu_{1n}, \mu_{2n})$ given by \eqref{x4} satisfies
$$
\mu_{1n}+\mu_{2n}=p_nJ_2(p_n)+(1-p_n)\inte \big(|\nabla u_{1n}|^2+|\nabla u_{2n}|^2\big)dx.
$$
By the fact $\mu_{1n}<\mu_{2n}<0$, we then obtain  that the sequence $\{\mu_{in}\}_n$ is bounded uniformly in $n>0$, $i=1, 2$. Set
$$
\lim\limits_{n\to\infty}\mu_{in}=\mu_i\leq0, \ \, i=1,2,$$
which then gives from \eqref{x4} that
\begin{eqnarray}\label{2.12V}
\left\{
\begin{array}{lll}
\!\!\!-u_1''-\big(u_1^2+u_2^2\big)u_1=\mu_1u_1  & \mathrm{in}\ \R,\\[1.5mm]
\!\!\!-u_2''-\big(u_1^2+u_2^2\big)u_2=\mu_2u_2  & \mathrm{in}\ \R.
\end{array}
\right.
\end{eqnarray}
Following \cite[Eq.~(44)]{i} and  \cite[Theorem~7]{ii}, since $(u_{1n},u_{2n})$ is a minimizer of $J_2(p_n)$, we get from  \eqref{2.7} that
\begin{equation}\label{2.14}
\mu_1+\mu_2
= \lim_{n\to\infty} \big(\mu_{1n}+\mu_{2n}\big)
= \lim_{n\to\infty} \frac{p_n+1}{3-p_n} J_2(p_n)
= 3J_2(2)=6J_1(2)<0,
\end{equation}
where
$$
J_1(2):=\inf\limits_{v\in H^1(\R),\, \|v\|_2^2=1}\big(\inte |v'|^2dx-\frac{1}{2}\inte v^4dx\big).
$$
To finish the proof of Lemma \ref{lem2.1} (ii), the rest is to prove   (\ref{mu}) and (\ref{J9}) in view of (\ref{2.12}) and (\ref{2.12V}).

We now claim that
\begin{equation}\label{2.15}
J_1(2)=-\frac 1 {48}.
\end{equation}
Actually,  applying the concentration-compactness principle  (see, e.g. \cite{con1, concen}), one can conclude that  $J_1(2)$ has a positive minimizer $v\in H^1(\R)$,  where $v>0$  satisfies the Euler-Lagrange equation
\begin{equation*}
-v''-v^3=\mu v\ \ \ \text{in}\ \, \R
\end{equation*}
for some suitable Lagrange multiplier $\mu<0$.
By the uniqueness of positive solutions for \eqref{1.3},  up to  translations, we then have
$v(x)=\sqrt{|\mu|}\,w\big(\sqrt{|\mu|}\,x\big)$ in $\R$,  and
\begin{equation}\label{2.17}
 1=\|v\|_2^2=\sqrt{|\mu|}\,\|w\|_2^2=4\sqrt{|\mu|},
\end{equation}
where $w>0$ is as in \eqref{a8}.
We thus obtain from above that
\begin{equation*}
 J_1(2)=\inte |v'|^2dx-\frac{1}{2}\inte v^4dx=-\frac 1 {48},
\end{equation*}
and the claim \eqref{2.15} is therefore proved.

We obtain from \eqref{2.14} and \eqref{2.15} that
\begin{equation}\label{J8}
\mu_1 \leq \mu_2\leq0 \ \ \ \text{and}\ \ \ \mu_1 + \mu_2 = -1/8.
\end{equation}
We next claim that $\mu_1 \le \mu_2 < 0$. On the contrary,
suppose
\begin{equation}\label{2:J8}
\lim_{n\to\infty} \mu_{2n} = \mu_2=0.
\end{equation}
We then deduce from \eqref{x4} that
\begin{equation*}
\int_{\mathbb{R}} (u'_{2n})^2dx
- \int_{\mathbb{R}} (u_{1n}^2+u_{2n}^2)^{p_n-1} u_{2n}^2dx
= \mu_{2n} \int_{\mathbb{R}} u_{2n}^2dx
\rightarrow 0 \ \ \hbox{as}\ \ n\to\infty,
\end{equation*}
and thus
\begin{align}\label{2.19}
J_2(p_n)
=&
\int_{\mathbb{R}} \bigl( |u_{1n}'|^2+|u_{2n}'|^2 \bigr)\,dx- \frac{1}{p_n} \int_{\mathbb{R}} (u_{1n}^2+u_{2n}^2)^{p_n}\,dx \nonumber\\
=& \int_{\mathbb{R}} |u_{1n}'|^2\,dx
- \frac{1}{p_n} \int_{\mathbb{R}} u_{1n}^{2p_n}\,dx\\
&
- \frac{1}{p_n} \int_{\mathbb{R}}
\bigl[(u_{1n}^2+u_{2n}^2)^{p_n}-u_{1n}^{2p_n}-p_n(u_{1n}^2+u_{2n}^2)^{p_n-1} u_{2n}^2\bigr]\,dx+o(1) \nonumber\\
&\ge J_1(p_n) + o(1)\ \ \hbox{as}\ \ n\to\infty,\nonumber
\end{align}
where we have used the convexity of the function $x\mapsto x^{p_n}$ for $p_n\ge 3/2$.
On the other hand,  the same argument of \eqref{2.7} gives from \cite[Lemma 12]{i} that
$$
\lim\limits_{n\to\infty}J_2(p_n) \leq \lim\limits_{n\to\infty}2J_1(p_n) =2J_1(2)<0,
$$
which however contradicts with \eqref{2.19}. Thus, (\ref{2:J8}) is false, and the claim $\mu_1 \le \mu_2 < 0$ is therefore proved in view of (\ref{J8}).

Since the above analysis gives that $\mu_1 \le \mu_2 < 0$ and $\mu_1 + \mu_2=-1/8 $, in order to prove  (\ref{mu}), it next suffices to show that
\begin{equation}\label{J12}
\mu_1=-1/16.
\end{equation}
We first prove that
\begin{equation}\label{2.61}
	\liminf_{n\to\infty}u_{1n}(z_n)>0,
\end{equation}
where $z_n\in \R$ is a global maximal point of  $u_{1n}(x)$ in $\R$.
Indeed, suppose \eqref{2.61} is false. By the uniform boundedness of $\{(u_{1n}, u_{2n})\}$ in $H^1(\R)\times H^1(\R)$,
we then derive from \eqref{x4} that
\begin{align}\label{2.59}
0>\lim\limits_{n\to\infty}\mu_{1n}=&\lim\limits_{n\to\infty}\inte |\nabla u_{1n}|^2dx-\lim\limits_{n\to\infty}\inte \big(u_{1n}^2+u_{2n}^2\big)^{p_n-1}\, u_{1n}^2dx\nonumber\\
\geq&\lim\limits_{n\to\infty}\inte |\nabla u_{1n}|^2dx
\geq0,\nonumber
\end{align}
a contradiction. This thus proves  \eqref{2.61}.

We now deduce from \eqref{2.61} that  there exist $\tilde{u}_1\in H^1(\R)\backslash\{0\}$ and  $\tilde{u}_2\in H^1(\R)$ such that
\begin{equation}\label{2.23}
\tilde{u}_{in}(x):=u_{in}(x+z_n)\rightharpoonup\tilde{u}_i\ \ \text{weakly\ in}\ H^1(\R)\ \ \text{as}\ \ n\to\infty,\ \   i=1,2,
\end{equation}
where $(\tilde{u}_1, \tilde{u}_2)$ solves the following elliptic system
\begin{eqnarray}\label{2.63b}
\left\{
\begin{array}{lll}
\!\!\!-\tilde{u}_1''-\big(\tilde{u}_1^2+\tilde{u}_2^2\big)\tilde{u}_1=\mu_1\tilde{u}_1  & \mathrm{in}\ \ \R,\\[1.5mm]
\!\!\!-\tilde{u}_2''-\big(\tilde{u}_1^2+\tilde{u}_2^2\big)\tilde{u}_2=\mu_2\tilde{u}_2  & \mathrm{in}\ \ \R.
\end{array}
\right.
\end{eqnarray}
Since $\tilde{u}_{1n}(x)>0$ in $\R$, the maximum principle (cf. \cite{elli}) yields that $\tilde{u}_1(x)>0$ in $\R$.
we next continue the proof by considering separately two cases: either $\tilde{u}_2(x)\equiv 0$ in $\R$ or $\tilde{u}_2(x)\not\equiv 0$ in $\R$.

For the case where $\tilde{u}_2(x)\equiv 0$ in $\R$,  similar to \eqref{2.17}, we obtain from \eqref{2.63b} that
\[
1=\liminf\limits_{n\to\infty}\inte\tilde{u}_{1n}^2dx\geq\inte \tilde{u}_1^2dx=4\sqrt{|\mu_1|},
\]
and thus $\mu_1\geq-\frac{1}{16}$. This thus yields from \eqref{J8}  that \eqref{J12} holds true, and  (\ref{mu}) is thus proved. For the case where $\tilde{u}_2(x)\not\equiv 0$ in $\R$, we claim that
\begin{equation}\label{2.25}
\mu_1=\mu_2.
\end{equation}
Indeed, since $\tilde{u}_{1n}(x)>0$ in $\R$,  if $\mu_1<\mu_2<0$, then  the  simplicity of the first eigenvalue for $-D_{xx}-\big(\tilde{u}_1^2+\tilde{u}_2^2\big)$ in $\R$ yields that there exists $x_0\in\R$ such that $\tilde{u}_2(x_0)=0$.
Applying the classification result of
\cite[Lemmas 15 \& 16]{ii},  we thus obtain from \eqref{2.63b} that
\begin{equation*}
1=\liminf\limits_{n\to\infty}\inte\tilde{u}_{1n}^2dx\geq\int_{\mathbb{R}} \tilde{u}_1^2 dx= 4\sqrt{|\mu_1|}.
\end{equation*}
This then gives from \eqref{J8} that $\mu_1=-1/16<\mu_2$, $i.e.,$ $\mu_1+\mu_2>-1/8$, which however contradicts with
\eqref{J8}. Thus, the claim (\ref{2.25}) is proved, which further implies that \eqref{J12} holds true. The also proves  (\ref{mu}).

We finally address the proof of (\ref{J9}). The above analysis shows that the limit $(u_1, u_2)$ of (\ref{2.12}) is a solution of the following elliptic system
\begin{equation}\label{J7}
\left\{
\begin{aligned}
- u_1'' - (u_1^2 + u_2^2) u_1 &= -\frac{1}{16} u_1\ \   \text{in}\ \, \R, \\
- u_2'' - (u_1^2 + u_2^2) u_2 &= -\frac{1}{16} u_2\ \   \text{in}\ \, \R.
\end{aligned}
\right.
\end{equation}
Since Lemma \ref{lem2.1} (i) gives that $u^2_1(x)+u^2_2(x)\not \equiv 0$ in $\R$, we obtain that one of the following three cases occurs:

Case 1. $u_1(x)\not\equiv0$ and $u_2(x)\equiv0$ in $\R$;

Case 2. $u_1(x)\equiv0$ and $u_2(x)\not\equiv0$ in $\R$;

Case 3. $u_1(x)\not\equiv0$ and $u_2(x)\not\equiv0$ in $\R$.

Because $u_{1n}(x)>0$ holds in $\R$,    the maximum principle yields that for Cases 1 \& 3, we have  $u_1(x)>0$ in $\R$. Applying the uniqueness of   positive solutions for  \eqref{1.3} and  the  simplicity of the first eigenvalue for $-D_{xx}-\big({u}_1^2+{u}_2^2\big)$ in $\R$,
we derive from (\ref{J7}) that there exists some point $ x_0\in\R$ such that
\begin{equation}\label{2.29}
\text{Case\ 1}: \ (u_1(x), u_2(x))=\big(w_*(x-x_0), \, 0\big) \ \  \text{in} \ \, \R\times \R,
\end{equation}
and
\begin{equation}\label{2.30a}
\text{Case\ 3}: \ u_2(x)=ku_1(x)=\frac{k}{\sqrt{1+k^2}}w_*(x-x_0)\ \   \text{in} \ \, \R\times \R
\end{equation}
holds for some  $k\in\R\backslash\{0\}$.
As for Case 2, we however obtain from (\ref{J7}) that
\begin{equation}\label{2.31}
u_2(x)=\frac{\sqrt{2}a_1}{e^{-\frac{1}{4}(x-x_0)}+4a_1^2e^{\frac{1}{4}(x-x_0)}}\ \   \text{holds for some}\ a_1\neq0\ \text{and}\ x_0\in\R,
\end{equation}
where we have used the classification result of \cite[Lemmas~15 \& 16]{ii}.

Since $u_{1n}^2+u_{2n}^2\to u_1^2+u_2^2$ strongly in $C^1_{\mathrm{loc}}(\mathbb{R})$ as $n\to\infty$,
it gives from (\ref{J10}) that $\big(u_1^2(x)+u_2^2(x)\big)'|_{x=0}=0$. This then yields  from \eqref{2.29}--\eqref{2.31} that
$$
x_0=0 \ \ \text{holds for Cases}\  1\ \&\ 3,
$$
and
\begin{equation}\label{2.32}
4a_1^2=e^{\frac{1}{2}x_0}\ \ \text{holds for Case}\  2.
\end{equation}
We thus conclude from above that  there exists a constant $a\in[-1,1]$ such that
\begin{equation*}
\big(u_1(x), u_2(x)\big)=\bigl( \sqrt{1-a^2}\, w_\ast(x),\; a\, w_\ast(x) \bigr)\ \ \ \text{in}\ \ \R\times \R,
\end{equation*}
$i.e.,$ (\ref{J9}) is proved. This completes the proof  of Lemma \ref{lem2.1}. \qed


Applying Lemma \ref{lem2.1}, we then establish the following convergence.

\begin{lem}\label{thm2.1}
Suppose $(u_{1n},u_{2n})$ is a minimizer of $J_2(p_n)$ satisfying \eqref{x4} and \eqref{J10}, where $p_n\in[3/2, 5/2]$ satisfies $p_n\to2$ as $n\to\infty$, and let $a\in[-1, 1]$ be given by Lemma \ref{lem2.1}.
Then we have
\begin{equation}\label{2-5}
\bar{u}_{1n}(x):=\sqrt{1-a^2}u_{1n}(x)+au_{2n}(x)\to w_*(x) \ \ \text{strongly\ in}\  L^r(\R)
\end{equation}
as $n\to\infty$,	 and
\begin{equation}\label{2-4}
\bar{u}_{2n}(x):=-au_{1n}(x)+\sqrt{1-a^2}u_{2n}(x)\to \pm w_*(x-\bar x_n)\ \ \text{strongly\ in}\  L^r(\R)
\end{equation}
as $n\to\infty$,  where $2\leq r<\infty$ is arbitrary,  and $ \bar x_n\in\R$ satisfying $| \bar x_n|\to\infty$ as $n\to\infty$ denotes the maximum point  of $|\bar u_{2n}(x)|$  in $\R$.
\end{lem}

\noindent{\bf Proof. }
Note from Lemma \ref{lem2.1} that
\begin{equation*}
(u_{1n},u_{2n}) \rightharpoonup \big(\sqrt{1-a^2}\,w_\ast, a w_\ast\big)\ \ \text{weakly in}\ H^1(\R) \ \text{as}\ \ n\to\infty.
\end{equation*}
We then get from (\ref{2-5}) that
\begin{equation*}
\bar u_{1n} \rightharpoonup w_\ast \ \ \text{weakly in}\ L^2(\R)\ \ \text{as} \ \ n\to\infty,\ \ \|\bar u_{1n}^2 \|_2 = \|w_\ast^2 \|_2=1.
\end{equation*}
This further implies that
\begin{equation}\label{2.30}
\bar u_{1n}\to w_\ast \quad \text{strongly in } L^2(\mathbb{R})\ \, \text{as} \ \, n\to\infty.
\end{equation}
Using the interpolation inequality and the uniform boundedness of $\{\hat{u}_{1n}\}$ in $H^1(\mathbb{R})$, we then obtain from  \eqref{2.30} that the convergence \eqref{2-5} holds true.

In the following, we prove the   convergence \eqref{2-4}.
We first claim that there exists a constant $\delta >0$, independent of $n>0$, such that
\begin{equation}\label{J13}
\liminf_{n\to\infty} |\bar u_{2n}(\bar{x}_n)| \ge \delta >0 ,
\end{equation}
where $\bar{x}_n$ is a global maximum  point of $|\bar u_{2n}(x)|$ in $\R$.
Indeed, suppose \eqref{J13} is false. It then follows from \eqref{x4} that
\begin{equation*}
\lim\limits_{n\to\infty}\mu_{2n}
= \lim\limits_{n\to\infty}\int_{\mathbb{R}} \Bigl[
(\bar u'_{2n})^2
- (\bar{u}_{1n}^2+\bar u_{2n}^2)^{p_n-1}\bar u_{2n}^2
\Bigr]\,dx=\inte (\bar u'_{2n})^2dx
\ge 0 ,
\end{equation*}
which however contradicts with \eqref{mu}. This proves the claim \eqref{J13}.

Set
\begin{equation*}
\hat u_{in}(x) := \bar u_{in}(x+\bar x_n),\ \ \ i=1, 2.
\end{equation*}
Following Lemma \ref{lem2.1}, we then  derive from \eqref{x4} and \eqref{J13} that  there exist $u_{10}\in H^1(\R)$  and  $u_{20}\in H^1(\R)\backslash\{0\}$ such that up to a subsequence  if necessary,
\begin{equation*}
(\hat u_{1n}, \hat u_{2n}) \rightharpoonup (u_{10},u_{20}) \ \ \text{weakly in}\ H^1(\R)\ \ \hbox{as}\,\ n\to\infty,
\end{equation*}
where the limit $(u_{10},u_{20})$ satisfies
\begin{equation*}
\left\{
\begin{aligned}
- u_{10}'' - (u_{10}^2+u_{20}^2)u_{10}
&= -\frac{1}{16} u_{10}\ \ \mbox{in}\ \, \R, \\
- u_{20}'' - (u_{10}^2+u_{20}^2)u_{20}
&= -\frac{1}{16} u_{20}\ \ \mbox{in}\ \, \R.
\end{aligned}
\right.
\end{equation*}
Since
\[
\bar{u}_{2n} \rightharpoonup0,\ \   \bar{u}_{2n}(x+\bar x_n) \rightharpoonup u_{20}\not \equiv 0  \quad \text{weakly in}\  H^1(\mathbb{R})\ \ \hbox{as}\,\ n\to\infty,
\]
we have
\[
\lim_{n \to \infty} |\bar x_n| = +\infty.
\]

We obtain from (\ref{2.30}) that $u_{10}(x)\equiv 0$ in $\R$, which implies that
\[
-u_{20}'' - u_{20}^3 = -\frac{1}{16} u_{20}\ \ \, \text{in}\ \, \R.
\]
Similar to \eqref{2.30a} and \eqref{2.32},  we thus get that 
\[
u_{20}(x) = \pm w_{*}(x-x_0)\ \ \ \text{in}\ \, \R\ \, \text{for some}\ x_0\in\R.
\]
Moreover, since
\[
\hat u_{2n}'(0)=0\ \  \text{and}\ \  \inte \hat{u}_{2n}^2dx = \inte u_{20}^2dx= \inte w_{*}^2dx =1,
\]
we derive  that $x_0=0$ and
\[
\hat{u}_{2n} \rightarrow \pm w_{*} \quad \text{strongly in}\  L^2(\mathbb{R})\ \ \text{as}\ \ n\to\infty.
\]
By the uniform boundedness of $\{\hat{u}_{2n}\}$ in $H^1(\mathbb{R})$, we further obtain that
\[
\hat{u}_{2n} \rightarrow \pm w_{*} \quad \text{strongly in }  L^r(\mathbb{R})\ \ \text{as}\ \ n\to\infty,\quad\forall\ r\geq2,
\]
$i.e.,$   \eqref{2-4} is proved. This therefore completes the proof of Lemma \ref{thm2.1}.
\qed

Employing Lemmas \ref{lem2.1} and \ref{thm2.1}, we next address the proof of Proposition \ref{J1}.

\vspace{0.05cm}
\noindent{\bf Proof of Proposition \ref{J1}.}  Under the assumptions of Proposition \ref{J1}, the limit (\ref{mun}) follows directly from Lemma \ref{lem2.1}.

In order to prove (\ref{J2}), let $(u_{1n},u_{2n})$ be a minimizer of $J_2(p_n)$ satisfying \eqref{x4}, where $p_n\in[3/2, 5/2]$ satisfies $p_n\to2$ as $n\to\infty$. It then yields from Lemma \ref{thm2.1} that there exist a constant $a_0 \in[-1,1]$ and a sequence $\{\bar{x}_n\}$ satisfying $|\bar{x}_n|\to\infty$ as $n\to\infty$ such that  
\begin{align}\label{2.34}
	&\text{either}\quad  (u_{1n},-u_{2n})\nonumber\\
	&\qquad\quad\,\to \big(\sqrt{1-a_0 ^2}w_{*}(x)+a_0 w_{*}(x-\bar x_n),\
	-a_0 w_{*}(x)+\sqrt{1-a_0 ^2}w_{*}(x-\bar x_n)\big),\\
	&	\text{or}\quad  (u_{1n},u_{2n})\nonumber\\ &\quad\ \,\to\big(\sqrt{1-a_0 ^2}w_{*}(x)-a_0 w_{*}(x-\bar x_n),\ a_0 w_{*}(x)+\sqrt{1-a_0 ^2}w_{*}(x-\bar x_n)\big)\label{2.35}
\end{align}
strongly in $L^{r}(\R)$ as $n\to\infty$, where $r\in[2, \infty)$ is arbitrary. Since $u_{1n}(x)>0$ in $\R$,  $a_0 \in[0,1]$ holds for \eqref{2.34} and $-a_0 \in[0,1]$ holds for \eqref{2.35}. Hence,
there exist a constant $a_0 \in[0,1]$ and a sequence $\{\bar{x}_n\}$ satisfying $\bar{x}_n\to\infty$ as $n\to\infty$ such that  up  to reflections (across the y-axis)  if necessary,
\begin{equation}\label{J4}
	(u_{1n},\, \pm u_{2n})
\to \big(\sqrt{1-a_0 ^2}w_{*}(x)+a_0 w_{*}(x-\bar x_n),\
	-a_0 w_{*}(x)+\sqrt{1-a_0 ^2}w_{*}(x-\bar x_n)\big)
\end{equation}
strongly in $L^{r}(\R)$ as $n\to\infty$, where $r\in[2, \infty)$ is arbitrary. Moreover, if $a_0 \in(\sqrt{2}/2, 1]$, then
\begin{equation}\label{2.37}
\begin{split}
	&\big(u_{1n}(-x+\bar x_n), \mp u_{2n}(-x+\bar x_n)\big)\\
	\to &\big(a_0 w_{*}(x)+\sqrt{1-a_0 ^2}w_{*}(x-\bar x_n),\
	-\sqrt{1-a_0 ^2}w_{*}(x)+ a_0w_{*}(x-\bar x_n)\big)
\end{split}
\end{equation}
strongly in $L^{r}(\R)$ as $n\to\infty$, where $r\in[2, \infty)$ is arbitrary. As a consequence of \eqref{J4} and \eqref{2.37},  we conclude there exist a constant $a_0 \in[0, \sqrt{2}/2]$ and a sequence $\{\bar{x}_n\}\subset \R$ satisfying $\bar x_n\to\infty$ as $n\to\infty$ such that up to translations and reflections if necessary,
\begin{equation}\label{J5}
	(u_{1n}, \pm u_{2n})\to \big(\sqrt{1-a_0 ^2}w_{*}(x)+a_0 w_{*}(x-\bar x_n),\
	-a_0 w_{*}(x)+\sqrt{1-a_0 ^2}w_{*}(x-\bar x_n)\big)
\end{equation}
strongly in $L^{r}(\R)$ as $n\to\infty$, where $r\in[2, \infty)$ is arbitrary.

Applying the  standard elliptic regularity theory, we obtain from \eqref{x4} that the sequence $\{(u_{1n},u_{2n})\}$ is uniformly bounded in $C^{1,\alpha}(\R)$ for some $\alpha\in(0,1)$.
We hence deduce from  the Arzel\`a--Ascoli theorem that  there exists a subsequence, still denoted by $\{(u_{1n},u_{2n})\}$, of $\{(u_{1n},u_{2n})\}$  such that for any fixed $R>0$, 
\begin{equation}\label{2.46}
	(u_{1n},\pm u_{2n})\to
	\big(\sqrt{1-a_0 ^2}w_{*}(x)-a_0 w_{*}(x-\bar x_n),
	-a_0 w_{*}(x)+\sqrt{1-a_0 ^2}w_{*}(x-\bar x_n)\big)
\end{equation}
strongly in $L^\infty\big(B_R(0)\cup B_R(\bar x_n)\big)$ as $n\to\infty$, where $a_0 \in[0, \sqrt{2}/2]$ and the sequence $\{\bar{x}_n\}\subset \R$ satisfying $\bar x_n\to\infty$ as $n\to\infty$ are as in (\ref{J5}).
On the other hand,  by the interpolation inequality
\[
\|f\|^2_{L^\infty[\alpha,\, \beta]}
\le 2\|f\|_{L^2[\alpha, \,\beta]}\|f'\|_{L^2[\alpha,\, \beta]}
+\frac{2}{|\alpha-\beta|}\|f\|^2_{L^2[\alpha,\, \beta]},
\]
we deduce from \eqref{J5} that for any $\varepsilon>0$, there exists a sufficiently large constant $R:=R(\varepsilon)>0$, which is independent of $n>0$, such that   for any  sufficiently large $n>0$,
\begin{equation}\label{2.47}
	\|u_{in}\|_{L^\infty\big(\R\setminus(B_R(0)\cup B_R(\bar x_n))\big)}< \varepsilon/2,\ \ i=1,2.
\end{equation}
As a consequence of \eqref{J5}--\eqref{2.47}, we obtain that  (\ref{J2}) holds true. This completes the proof of Proposition \ref{J1}. \qed

\section{Limiting Behavior of  Minimizers as $p\to 2$}
The main purpose of this section is to analyze the limiting profiles of minimizers for $J_2(p_n)$, where $p_n\in[3/2, 5/2]$ satisfies $p_n\to2$ as $n\to\infty$. Towards this purpose, let $(\bar u_{1n}, \bar u_{2n})$ be a minimizer of $J_2(p_n)$, where $p_n\in[3/2, 5/2]$ satisfies $p_n\to2$ as $n\to\infty$. Without loss of generality, applying Proposition \ref{J1}, we can suppose that
\begin{align}\label{J2a}
(\bar u_{1n}, \bar u_{2n})\to \big(\sqrt{1-a_0^2}w_{*}(x)+a_0w_{*}(x-\bar x_n),\ -a_0w_{*}(x)+\sqrt{1-a_0^2}w_{*}(x-\bar x_n)\big)
\end{align}
strongly in $L^{r}(\R)$ as $n\to\infty$, where $r\in[2, \infty]$ is arbitrary, $a_0\in[0,\sqrt{2}/2]$ and $\bar x_n\geq0$ satisfying $\lim_{n\to\infty}\bar x_n=\infty$ are as in Proposition \ref{J1}.
Let $y_n\in\R$ be a global maximum point of $\bar u_{1n}>0$ in $\R$. We then obtain from (\ref{J2a}) that $\lim\limits_{n\to\infty}y_n=0$, and
\begin{equation}\label{2.39}
\begin{split}
&\big(u_{1n}(x),  u_{2n}(x)\big):=\big(\bar u_{1n}(x+y_n),\bar u_{2n}(x+y_n)\big)\\
\to& \big(\sqrt{1-a_0^2}w_{*}(x)+a_0w_{*}(x-\bar x_n),\
-a_0w_{*}(x)+\sqrt{1-a_0^2}w_{*}(x-\bar x_n)\big)
\end{split}
\end{equation}
strongly in $L^{r}(\R)$ as $n\to\infty$, where $r\in[2, \infty]$ is arbitrary, $a_0\in[0, \sqrt{2}/2]$ and $\bar x_n\in\R$ satisfying $\lim\limits_{n\to\infty}\bar x_n=\infty$ are as in Proposition \ref{J1}. We further suppose that $x_n\in\R$ is a global maximum point of $u_{2n}$ defined by (\ref{2.39}). We then deduce from \eqref{2.39} that $\lim\limits_{n\to\infty}(x_n-\bar x_n)=0$, which further yields that
\begin{align}\label{3.2}
  u_{1n}(0)=\max\limits_{x\in\R}  u_{1n}(x)>0, \ \   u_{2n}(x_n)=\max\limits_{x\in\R} u_{2n}(x),\ \   \lim\limits_{n\to\infty}x_n=\infty,
\end{align}
 and
\begin{align}\label{2.40}
&(u_{1n},   u_{2n})
\to \big(\sqrt{1-a_0^2}w_{*}(x)+a_0w_{*}(x- x_n),\
-a_0w_{*}(x)+\sqrt{1-a_0^2}w_{*}(x- x_n)\big)
\end{align}
strongly in $L^{r}(\R)$ as $n\to\infty$, where $a_0\in[0, \sqrt{2}/2]$, and $r\in[2, \infty]$ is arbitrary. In this section, we mainly analyze  in Lemma \ref{a} the limiting behavior of $(u_{1n}, u_{2n})$ as $n\to\infty$.

Denote
\begin{align}\label{w}
\widetilde w_n(x):=\Big(\frac{2\sqrt{p_n}}{e^{(p_n-1)x}
+e^{(1-p_n)x}}\Big)^{\frac{1}{p_n-1}}>0,
\end{align}
so that  $\widetilde w_n$ is the unique (up to translations) positive solution of the following equation
\begin{equation*}\label{4.4}
\widetilde w_n''- \widetilde w_n+\widetilde w_n^{2p_n-1}=0 \ \  \mathrm{in} \,\ \R,\ \  1<p_n<3.
\end{equation*}
This yields that for $i=1, 2,$ if $\widetilde w_{in}(x)>0$ is a unique positive solution of
\begin{equation}\label{AA3.1}
\widetilde w''_{in}+\mu_{in}\widetilde w_{in}+\widetilde w_{in}^{2p_n-1}=0\ \ \mathrm{in} \,\ \R,
\end{equation}
where $\mu_{1n}<\mu_{2n}<0$ and $3/2\leq p_n\leq5/2$ are as in Proposition \ref{J1}, then up to translations,
\begin{equation}\label{win}
\widetilde w_{in}(x):=|\mu_{in}|^{\frac{1}{2(p_n-1)}}\widetilde w_n\big(\sqrt{|\mu_{in}|}x\big)
=\Big(\frac{2\sqrt{|\mu_{in}|p_n}}{e^{(p_n-1)\sqrt{|\mu_{in}|}x}
+e^{(1-p_n)\sqrt{|\mu_{in}|}x}}\Big)^{\frac{1}{p_n-1}}>0.
\end{equation}
Moreover, for any $n\in\mathbb{N}^+$, define $  F_n(a, b, \delta, \eta)\in C^1\big( [-4/5, \, 4/5]^4,\,  \R\big)$ by
\begin{equation*}\label{f}
\begin{split}
F_n(a, b, \delta, \eta):= \inte \big[u_{1n}(x)-\sqrt{1-b^2}\widetilde w_{1n}(x-\delta)-a\widetilde w_{2n}(x- x_n+\eta)\big]^2dx\\
+\inte\big[ u_{2n}(x)-b\widetilde w_{1n}(x-\delta)-\sqrt{1-a^2}\widetilde w_{2n}(x- x_n+\eta)\big]^2 dx>0,
\end{split}
\end{equation*}
where $(u_{1n}, u_{2n})$ and $x_n\in\R$ are as in \eqref{2.40}.
It then yields from \eqref{2.40} that for any sufficiently large $n>0$, the minimum  of $F_n(\cdot )$ is attained at $(a_n, b_n, \delta_n, \eta_n)\in(-4/5, \, 4/5)^4$ satisfying
\begin{align}\label{3.5}
  \lim\limits_{n\to\infty}(a_n, b_n, \delta_n, \eta_n)=(a_0, -a_0, 0, 0),
\end{align}
where $a_0\in[0, \sqrt{2}/2]$ is as in \eqref{2.40}.

Set
\begin{eqnarray}\label{4.10a}
	\left\{
	\begin{array}{lll}
		\!	\!	\!	 u_{1n}(x)=\sqrt{1-b^2_n}\widetilde w_{1n}(x-\delta_n)+a_n\widetilde w_{2n}(x- x_n+\eta_n)+ \phi_n(x), \\[1mm]
		\!	\!	\!	 u_{2n}(x)=b_n\widetilde w_{1n}(x-\delta_n)+\sqrt{1-a_n^2}\widetilde w_{2n}(x- x_n+\eta_n)+ \psi_n(x),
	\end{array}
	\right.
\end{eqnarray}
where $(a_n, b_n, \delta_n, \eta_n)\in(-4/5, \, 4/5)^4$ is as in \eqref{3.5}, 
and  $ x_n$ satisfying $\lim _{n\to\infty}x_n=\infty$ is as in \eqref{2.40}. By the definition of $(a_n, b_n, \delta_n, \eta_n)$, it then follows from \eqref{2.40} and \eqref{3.5} that
\begin{eqnarray}\label{2.42}
	\phi_n, \ \psi_n\to0\ \  \text{strongly in}\ L^2(\R)\cap L^\infty(\R)\ \  \text{as}\ \ n\to\infty,
\end{eqnarray}
\begin{eqnarray}\label{3.7}
\left\{
\begin{array}{lll}
\!\!\!0=-\sqrt{1-b_n^2}\,\displaystyle\frac{\partial F_n}{\partial b}\Big|_{(a_n, b_n, \delta_n, \eta_n)}\\[2mm]
\ =2\displaystyle\inte \big[-b_n \phi_n+\sqrt{1-b_n^2}\psi_n\big]\widetilde w_{1n}(x-\delta_n)dx,\\[4mm]
\!\!\! 0=\displaystyle\frac{\partial F_n}{\partial \delta}\Big|_{(a_n, b_n, \delta_n, \eta_n)}=2\displaystyle\inte \big[\sqrt{1-b_n^2} \phi_n+b_n\psi_n\big]\widetilde w'_{1n}(x-\delta_n)dx,
\end{array}
\right.
\end{eqnarray}
and
\begin{eqnarray}\label{3.8}
\left\{
\begin{array}{lll}
\!\!\!0=-\sqrt{1-a_n^2}\,\displaystyle\frac{\partial F_n}{\partial a}\Big|_{(a_n, b_n, \delta_n, \eta_n)}\\[2mm]
\ =2\displaystyle\inte \big[\sqrt{1-a_n^2} \phi_n-a_n\psi_n\big]\widetilde w_{2n}(x- x_n+\eta_n)dx,\\[4mm]
\!\!\!0=-\displaystyle\frac{\partial F_n}{\partial \eta}\Big|_{(a_n, b_n, \delta_n, \eta_n)}=2\displaystyle\inte \big[a_n \phi_n+\sqrt{1-a_n^2}\psi_n\big]\widetilde w'_{2n}(x- x_n+\eta_n)dx.
\end{array}
\right.
\end{eqnarray}
Denote
\begin{eqnarray}\label{4.6}
\left\{
\begin{array}{lll}
\!	\!	\!	 \hat u_{1n}(x):=\sqrt{1-b^2_n} u_{1n}(x)+b_n  u_{2n}(x), \\[1mm]
\!	\!	\!	 \hat u_{2n}(x):=a_n u_{1n}(x)+\sqrt{1-a_n^2} u_{2n}(x),
\end{array}
\right.
\end{eqnarray}
which then yields from \eqref{4.10a} that
\begin{eqnarray}\label{4.10}
\left\{
\begin{array}{lll}
\!	\!	\!	 \hat u_{1n}(x)= \widetilde w_{1n}(x-\delta_n)+K_n \widetilde w_{2n}(x-x_n+\eta_n)+\hat \phi_n(x), \\[1mm]
\!	\!	\!	 \hat u_{2n}(x)= K_n \widetilde w_{1n}(x-\delta_n)+\widetilde w_{2n}(x- x_n+\eta_n)+\hat \psi_n(x),
\end{array}
\right.
\end{eqnarray}
where  $K_n:=a_n\sqrt{1-b^2_n}+b_n\sqrt{1-a^2_n}$, and
\begin{eqnarray}\label{AA3.2}
	\left\{
	\begin{array}{lll}
		\!	\!	\!	 \hat  \phi_n(x):=\sqrt{1-b^2_n} \phi_n(x)+b_n \psi_n(x), \\[1mm]
		\!	\!	\!	   \hat \psi_n(x):=a_n \phi_n(x)+\sqrt{1-a^2_n}\psi_n(x).
	\end{array}
	\right.
\end{eqnarray}
We thus conclude from \eqref{3.5}, \eqref{2.42} and \eqref{4.10} that
\begin{equation}\label{4.8a}
\big(\hat u_{1n},\ \hat u_{2n}(x+ x_n)\big)\to (w_*,\, w_*)\ \ \text{strongly in}\   L^2(\R)\cap L^\infty(\R) \ \, \text{as}\ \, n\to\infty.
\end{equation}

We now denote
\begin{equation}\label{3.16a}
w_{1n}(x):=\widetilde w_{1n}(x-\delta_n),\ \ w_{2n}(x):=\widetilde w_{2n}(x- x_n+\eta_n).
\end{equation}
Note from  \eqref{x4} and \eqref{4.6} that
\begin{eqnarray}\label{x1}
	\left\{
	\begin{array}{lll}
		\!\!\!-\hat u''_{1n}-\big( u_{1n}^2+ u_{2n}^2\big)^{p_n-1}\, \hat u_{1n}=\mu_{1n}\hat u_{1n}+b_nu_{2n}(\mu_{2n}-\mu_{1n})& \mathrm{in} \,\ \R,\\[1.5mm]
		\!\!\!-\hat u''_{2n}-\big( u_{1n}^2+u_{2n}^2\big)^{p_n-1}\, \hat u_{2n}=\mu_{2n}\hat u_{2n}  -a_nu_{1n}(\mu_{2n}-\mu_{1n})& \mathrm{in} \,\ \R.
	\end{array}
	\right.
\end{eqnarray}
We derive from  \eqref{AA3.1}, \eqref{4.10} and \eqref{x1} that
\begin{equation}\label{AA3.5}
	\begin{aligned}
&-\hat \phi''_n-\mu_{1n}\hat \phi_n-({u}_{1n}^2+{u}_{2n}^2)^{p_n-1}\hat \phi_n\\
=&w_{1n}\big[({u}_{1n}^2+{u}_{2n}^2)^{p_n-1}-w_{1n}^{2(p_n-1)}\big] +K_nw_{2n}\big[({u}_{1n}^2+{u}_{2n}^2)^{p_n-1}-w_{2n}^{2(p_n-1)}\big]\\
&-b_n^2(\mu_{1n}-\mu_{2n})w_{1n}+a_n\sqrt{1-b^2_n}(\mu_{1n}-\mu_{2n})w_{2n}-b_n(\mu_{1n}-\mu_{2n})\psi_{n},
\end{aligned}
\end{equation}
and
\begin{equation}\label{AA3.6}
\begin{aligned}
&-\hat \psi''_n-\mu_{2n}\hat \psi_n-({u}_{1n}^2+{u}_{2n}^2)^{p_n-1}\hat \psi_n\\
=&w_{2n}\big[({u}_{1n}^2+{u}_{2n}^2)^{p_n-1}-w_{2n}^{2(p_n-1)}\big]+K_nw_{1n}\big[({u}_{1n}^2+{u}_{2n}^2)^{p_n-1}-w_{1n}^{2(p_n-1)}\big]\\
&-b_n\sqrt{1-a^2_n}(\mu_{1n}-\mu_{2n})w_{1n}+a^2_n(\mu_{1n}-\mu_{2n})w_{2n}+a_n(\mu_{1n}-\mu_{2n})\phi_{n},
\end{aligned}
\end{equation}
where
\begin{align}\label{4.27b}
K_n:=\big(a_n\sqrt{1-b^2_n}+b_n\sqrt{1-a^2_n}\big),
\end{align}
and
\begin{align}\label{4.27a}
&\ \ \ \ ({u}_{1n}^2+{u}_{2n}^2)\nonumber\\
&=w_{1n}^2+w_{2n}^2+2\big(\sqrt{1-b^2_n}w_{1n}+a_nw_{2n}\big) \phi_n+2\big(b_nw_{1n}+\sqrt{1-a^2_n}w_{2n}\big) \psi_n\nonumber\\
&\ \ \ \ +2\big(\sqrt{1-b^2_n}\, a_n+\sqrt{1-a^2_n}\, b_n\big)w_{1n}w_{2n}+ \phi_n^2+\psi_n^2\\
&=w_{1n}^2+w_{2n}^2+2w_{1n}\hat \phi_n+2w_{2n}\hat \psi_n+2K_nw_{1n}w_{2n}+ (\phi_n^2+\psi_n^2).\nonumber
\end{align}
We start with the following $L^\infty$-estimates.

\begin{lem}\label{O2}
Denote
\begin{eqnarray}\label{O1}
\left\{
\begin{array}{lll}
\!	\!	\!
F_{1n}:=w_{1n}\big[({u}^2_{1n}+{u}^2_{2n})^{p_n-1}-w_{1n}^{2(p_n-1)}\big]-2(p_n-1)w_{1n}^{2p_n-2}\hat\phi_n,\\[2mm]
\!	\!	\! F_{2n}:=w_{2n}\big[({u}^2_{1n}+{u}^2_{2n})^{p_n-1}-w_{2n}^{2(p_n-1)}\big]-2(p_n-1)w_{2n}^{2p_n-2}\hat\psi_n.
\end{array}
\right.
\end{eqnarray}
Then we have for $i=1, 2,$
\begin{equation}\label{O3a}
\begin{split}
\|F_{in}\|_{\infty}
\leq C\Big( e^{-\sqrt{|\mu_{in}|}\, x_n}+|K_n|e^{-\sqrt{|\mu_{2n}|}\, x_n}\Big)
+o\big(	\|\hat\phi_n\|_{\infty}+\|\hat\psi_n\|_{\infty}\big)
\ \ \text{as}\ \ n\to\infty,
\end{split}\end{equation}
where $K_n$ is as in (\ref{4.27b}), and the constant $C>0$ is independent of $n>0$.
\end{lem}

\noindent \textbf{Proof.} By the definition of $w_{in}(x)$, we note from \eqref{2.42} that
$$
\|\hat\phi_n\|_\infty=o(1),\ \ \ \|\hat\psi_n\|_\infty=o(1)\ \ \ \text{as}\ \ n\to\infty,
$$
and
\begin{eqnarray}\label{e}
\left\{
\begin{array}{lll}
\!	\!	\!
C_1e^{-\sqrt{|\mu_{1n}|}|x|}\leq w_{1n}(x)\leq C_2e^{-\sqrt{|\mu_{1n}|}|x|}\ \ \ \text{in}\ \, \R,\\[1.5mm]
\!	\!	\!
C_1e^{-\sqrt{|\mu_{2n}|}|x-x_n|}\leq w_{2n}(x)\leq C_2e^{-\sqrt{|\mu_{2n}|}|x-x_n|}\ \ \ \text{in}\ \, \R.
\end{array}
\right.
\end{eqnarray}
where $ 0<C_1<C_2$ are independent of $n>0$,  $\mu_{1n}<\mu_{2n}<0$ satisfy $\lim\limits_{n\to\infty}\mu_{1n}=\lim\limits_{n\to\infty}\mu_{2n}=-1/16$.
We then calculate from \eqref{4.27a} that
\begin{equation}\label{4.32b}
\begin{split}
|F_{1n}(x)|+&|F_{2n}(x)|=o\big( e^{-\sqrt{|\mu_{2n}|}x_n}+\|\hat \phi_n\|_\infty+\|\hat\psi_n\|_\infty\big)e^{-\frac{1}{4}\sqrt{|\mu_{2n}|}x_n}\\
 &\mbox{uniformly in} \ \ \big(-\infty,\, -\frac{x_n}{2}\big]\cup\big[\frac{3x_n}{2},\, \infty\big)
 \ \ \mbox{as}\ \ n\to\infty, \end{split}
\end{equation}
\begin{align}\label{4.33}
|F_{1n}(x)|
=&\big|w_{1n}(x)\big[\big(w^2_{2n}(x)+h_{1n}(x)\big)^{p_n-1}-w_{1n}^{2(p_n-1)}(x)\big]-2(p_n-1)w_{1n}^{2p_n-2}(x)\hat\phi_n(x)\big|\nonumber\\
\leq&2^{p_n-1}w_{1n}(x)w_{2n}^{2(p_n-1)}(x)+2^{p_n-1}w_{1n}(x)h_{1n}^{p_n-1}(x)\\
&+o\big( e^{-\sqrt{|\mu_{2n}|}x_n}+\|\hat\phi_n\|_\infty\big)e^{-\frac{1}{4}\sqrt{|\mu_{2n}|}x_n}\nonumber\\
\leq& Ce^{-\sqrt{|\mu_{1n}|}x_n}+o\big(\|\hat\phi_n\|_\infty+\|\hat\psi_n\|_\infty\big)\ \  \text{uniformly in} \ \   \big[\frac{x_n}{2}, \frac{3x_n}{2}\big] \ \ \text{as}\ \ n\to\infty,\nonumber
\end{align}
and
\begin{align}\label{4.34}
|F_{2n}(x)|
=&\Big|w_{2n}(x)\big[\big(w^2_{1n}(x)+h_{2n}(x)\big)^{p_n-1}-w_{2n}^{2(p_n-1)}(x)\big]-2(p_n-1)w_{2n}^{2p_n-2}(x)\hat\psi_n(x)\Big|\nonumber\\
\leq& Ce^{-\sqrt{|\mu_{2n}|}x_n}+o\big(\|\hat\phi_n\|_\infty+\|\hat\psi_n\|_\infty\big)\ \ \ \text{uniformly in} \ \   [-\frac{x_n}{2}, \frac{x_n}{2}]
\end{align}
as $n\to\infty$, where
\begin{align}\label{4.35}
h_{1n}(x):=w_{1n}^2+2w_{1n}\hat\phi_n+2w_{2n}\hat\psi_n+2K_nw_{1n}w_{2n}+\phi_n^2+\psi_n^2=o(1)
\end{align}
uniformly in $\big[\frac{x_n}{2}, \frac{3x_n}{2}\big]$  as $n\to\infty$, and
\begin{align}\label{4.35a}
h_{2n}(x):=w_{2n}^2+2w_{1n}\hat\phi_n+2w_{2n}\hat\psi_n+2K_nw_{1n}w_{2n}+\phi_n^2+\psi_n^2=o(1)
\end{align}
uniformly in $\big[-\frac{x_n}{2}, \frac{x_n}{2}\big] $  as $n\to\infty$.

Moreover, denote
$$
\Omega_{1n}:=\Big\{x\in \big[-\frac{x_n}{2}, \frac{x_n}{2}\big] : \ |h_{2n}(x)|/w_{1n}^2(x)\leq1/2\Big\}, \ \   n\in\mathbb{N}^+.$$
One can verify that
\begin{align*}
|F_{1n}(x)|
=&\Big|w_{1n}(x)\big[\big(w^2_{1n}(x)+h_{2n}(x)\big)^{p_n-1}-w_{1n}^{2(p_n-1)}(x)\big]-2(p_n-1)w_{1n}^{2p_n-2}(x)\hat\phi_n(x)\Big|\nonumber\\
\leq& C\Big[|h_{2n}^{p_n-\frac{1}{2}}(x)|+|h_{2n}^{p_n-1}(x)|\, |\hat\phi_n(x)|\Big],\ \ \ \forall\ x\in \big[-\frac{x_n}{2}, \frac{x_n}{2}\big] \backslash\Omega_{1n},
\end{align*}
and
\begin{align*}
&|F_{1n}(x)|\nonumber\\
=&\Big|(p_n-1)w_{1n}^{2p_n-3}h_{2n}\Big[1+\frac{p_n-2}{2}\big(w^2_{1n}+\theta_nh_{2n}\big)^{p_n-3}h_{2n}w_{1n}^{-2p_n+4}\Big]-2(p_n-1)w_{1n}^{2p_n-2}\hat\phi_n\Big|\nonumber\\
=&\Big|(p_n-1)w_{1n}^{2p_n-3}h_{2n}\Big[1+\frac{p_n-2}{2}\frac{1}{\big(1+\frac{\theta_{1n}h_{2n}}{w_{1n}^2})^{3-p_n}}\frac{h_{2n}}{w_{1n}^2}\Big]-2(p_n-1)w_{1n}^{2p_n-2}\hat\phi_n\Big| \nonumber\\
\leq&(p_n-1)\Big|w_{1n}^{2p_n-3}h_{2n}-2w_{1n}^{2p_n-2}\hat\phi_n\Big|
+2|p_n-2|\big|w_{1n}^{2p_n-3}h_{2n}\big|,\ \ \ \forall\ x\in \Omega_{1n},
\end{align*}
where  $\theta_{1n}\in(0,1)$.
These thus imply from \eqref{e} that
\begin{align}\label{4.25b}
|F_{1n}(x)|\leq& (p_n-1)\Big|w_{1n}^{2p_n-3}(x)h_{2n}(x)-2w_{1n}^{2p_n-2}(x)\hat\phi_n(x)\Big|
+2|p_n-2|\big|w_{1n}^{2p_n-3}(x)h_{2n}(x)\big|\nonumber\\
&+C\Big[|h_{2n}^{p_n-\frac{1}{2}}(x)|+|h_{2n}^{p_n-1}(x)|\hat\phi_n(x)|\Big]\nonumber\\
\leq&2\Big[w_{1n}^{2p_n-3}(x)w_{2n}^2(x)+2|K_n|w_{1n}^{2(p_n-1)}(x)w_{2n}(x)\Big]\\
& +o\big(e^{-\frac{5}{4}\sqrt{|\mu_{2n}|}x_n}+|\hat\phi_n(x)|+|\hat\psi_n(x)|\big)\nonumber\\
\leq&C\Big(e^{-\frac{5}{4}\sqrt{|\mu_{2n}|}x_n}+K_ne^{-\sqrt{|\mu_{2n}|}x_n}\Big)+o\big(\|\hat\phi_n\|_\infty+\|\hat\psi_n\|_\infty\big)\nonumber
\end{align}
uniformly in $\big[-\frac{x_n}{2}, \frac{x_n}{2}\big]$ as $n\to\infty$, where we have used the facts that
\begin{align*}
0\leq& w_{1n}^{2p_n-3}(x)w_{2n}^2(x)
\leq\frac{3}{2} e^{-(2p_n-3)\sqrt{|\mu_{1n}|}x}e^{-2\sqrt{|\mu_{2n}|}(x_n-x)}\\
\leq&\frac{3}{2} e^{-2\sqrt{|\mu_{2n}|}x_n}e^{(5-2p_n)\sqrt{|\mu_{2n}|}x}
\leq\frac{3}{2} e^{-2\sqrt{|\mu_{2n}|}x_n}e^{\frac{5-2p_n}{2}\sqrt{|\mu_{2n}|}x_n}\ \ \ \text{in}\ \ [0, x_n/2],
\end{align*}
and
\begin{align*}
0\leq& w_{1n}^{2p_n-3}(x)w_{2n}^2(x)
\leq\frac{3}{2} e^{(2p_n-3)\sqrt{|\mu_{1n}|}x}e^{-2\sqrt{|\mu_{2n}|}(x_n-x)}
\leq \frac{3}{2}e^{-2\sqrt{|\mu_{2n}|}x_n} \ \ \text{in}\ \ [-x_n/2, 0].
\end{align*}
Similar to \eqref{4.25b}, we have
\begin{align}\label{4.25}
	|F_{2n}(x)|\leq C\Big(e^{-\frac{5}{4}\sqrt{|\mu_{2n}|}x_n}+K_ne^{-\sqrt{|\mu_{2n}|}x_n}\Big)+o\big(\|\hat\phi_n\|_\infty+\|\hat\psi_n\|_\infty\big)
\end{align}
uniformly in $\big[\frac{x_n}{2}, \frac{3x_n}{2}\big]$ as $n\to\infty$.
We thus conclude from \eqref{4.32b}--\eqref{4.25} that \eqref{O3a} holds true, which therefore completes the proof of Lemma \ref{O2}. \qed

Applying Lemma \ref{O2}, we now analyze the following  $L^\infty$-estimates of $\hat\phi_n$ and $\hat\psi_n$ as $n\to\infty$.

\begin{lem}\label{NL1}
There exists a constant $C>0$, which is independent  of $n>0$,  such that
\begin{eqnarray*}
\|\hat \phi_n\|_\infty+\|\hat \psi_n\|_\infty\leq C\Big[e^{-\sqrt{|\mu_{2n}|}\, x_n}+\big(b_n^2+|a_n|\sqrt{1-b_n^2}\big)|\mu_{2n}-\mu_{1n}|\Big]\ \   \text{as}\ \ n\to\infty,
\end{eqnarray*}
where $a_n$ and $b_n$ are given in \eqref{3.5}, and $x_n$ satisfying $\lim\limits_{n\to\infty} x_n=\infty$ is as in \eqref{3.2}.
\end{lem}

\noindent \textbf{Proof.}  On the contrary, suppose that
\begin{align*}
	\lim\limits_{n\to\infty}\frac{\|\hat  \phi_n\|_\infty+\|\hat  \psi_n\|_\infty} {e^{-\sqrt{|\mu_{2n}|}\, x_n}+\big(b_n^2+|a_n|\sqrt{1-b_n^2}\big)|\mu_{2n}-\mu_{1n}|}=\infty.
\end{align*}
Denote
\begin{align*}
	\widetilde {\phi}_n:=\frac{\hat \phi_n} {\|\hat \phi_n\|_\infty+\|\hat\psi_n\|_\infty},\ \ \ \widetilde {\psi}_n:=\frac{\hat \psi_n} {\|\hat\phi_n\|_\infty+\|\hat\psi_n\|_\infty}.
\end{align*}
Applying Lemma \ref{O2}, it then follows from \eqref{AA3.5} and \eqref{AA3.6} that
\begin{align}\label{4.27}
	&-\widetilde {\phi}''_n-\mu_{1n}\widetilde {\phi}_n-({u}_{1n}^2+{u}_{2n}^2)^{p_n-1}\widetilde {\phi}_n-2(p_n-1)w_{1n}^{2p_n-2}\widetilde {\phi}_n\\
	 =&o(|\widetilde {\phi}_n|+|\widetilde {\psi}_n|)+O\Big(\frac{e^{-\sqrt{|\mu_{2n}|}\, x_n}+(b_n^2+|a_n|\sqrt{1-b_n^2})|\mu_{2n}-\mu_{1n}|}{\|\hat\phi_n\|_\infty+\|\hat\psi_n\|_\infty}\Big)\ \ \text{in}\ \, \R \ \   \text{as}\ \ n\to\infty, \nonumber
\end{align}
where $f_n:=O(A_n)$ means that  there exists a constant $C>0$, independent of $n>0$, such that $|f_n|\leq C A_n$ holds for any sufficiently large $n>0$. Since $\|\widetilde {\phi}_n\|_\infty+\|\widetilde {\psi}_n\|_\infty \equiv 1$, without loss of generality, we suppose that
\begin{align}\label{4.29a}
	\widetilde {\phi}_n(y_n)=\|\widetilde {\phi}_n\|_\infty\geq1/2,
\end{align}
where $y_n\in \R$ is a global maximum point of $\widetilde {\phi}_n$ in $\R$.

We now consider the first case where  both $\{y_n\}$  and  $\{y_n - x_n\}$ are unbounded in $n>0$.  In this case, it follows from \eqref{4.27}  that
\begin{align}\label{4.41}
	-\widetilde {\phi}''_n-\mu_{1n}\widetilde {\phi}_n=o(1)
	\ \ \ \text{in}\ B_1(y_n) \ \ \text{as}\ \ n\to\infty,
\end{align}
which however contradicts with the fact that
\begin{align}\label{4.42a}
	-\widetilde {\phi}''_n(y_n)-\mu_{1n}\widetilde {\phi}_n(y_n)\geq-\mu_{1n}\widetilde {\phi}_n(y_n)\geq\frac{|\mu_{1n}|}{2}=1/32+o(1) \ \ \text{as}\ \ n\to\infty.
\end{align}
This then completes the proof of Lemma \ref{NL1} for the first case.

We then consider the second case where  $\{y_n\}$ is bounded uniformly in $n>0$.
In this case, we note from \eqref{4.27}  that
\begin{eqnarray}\label{4.30}
-\widetilde {\phi}''_n-\mu_{1n}\widetilde{\phi}_n-(2p_n-1)w_{1n}^{2p_n-2}\widetilde{\phi}_n=o(1)\ \ \ \text{in}\ B_{x_n/4}(0)\ \ \text{as}\ \ n\to\infty. 
\end{eqnarray}
Since  $\|\widetilde{\phi}_n\|_\infty\leq1$, the standard elliptic regularity  gives  from the above that $\{\widetilde{\phi}_n\}$ is bounded uniformly in $H^2(\R)$. This then implies that 
 up to a subsequence  if necessary,
\begin{equation}\label{4.29}
	\widetilde{\phi}_n\to \widetilde{\phi}\ \ \ \text{strongly in}\ C^1_{loc}(\R)\ \text{as}\ n\to\infty\ \ \text{for some}\ \widetilde{\phi}\in C^1_{loc}(\R),
\end{equation}
and
\begin{equation}\label{4.31}
-\widetilde {\phi}''+\frac{1}{16}\widetilde{\phi}-3w_*^2\widetilde{\phi}=0\ \   \text{in}\ \, \R,
\end{equation}
where $w_*(x)>0$ is as in \eqref{2.8b}.
Since $\text{ker}\big(-D_{xx}+\frac{1}{16}-3w_*^2\big)=\{w_*'\}$,  we deduce from \eqref{4.29} that
\begin{align}\label{4.32}	
	\widetilde{\phi}=c_1w'_* \ \ \ \text{holds for some}\ \  c_1\in\R.
\end{align}
Moreover, one can derive from \eqref{3.7} and \eqref{AA3.2} that
\begin{equation}\label{3.39}
\begin{split}
0=&\lim\limits_{n\to\infty}\inte \frac{\sqrt{1-b_n^2} \phi_n+b_n\psi_n}{\|\hat\phi_n\|_\infty+\|\hat\psi_n\|_\infty}w'_{1n}dx
=\lim\limits_{n\to\infty}\inte 	\widetilde\phi_n w'_{1n}dx \\
=&\lim\limits_{n\to\infty}\inte \widetilde\phi_n w'_*dx=\inte 	\widetilde\phi w'_*dx.
\end{split}
\end{equation}
We then conclude from \eqref{4.32} and \eqref{3.39} that $c_1=0$, i.e.,
\begin{align}\label{3.40}
	\widetilde{\phi}(x)\equiv0\ \ \text{ in}\ \,\R,
\end{align}
which however contradicts with \eqref{4.29a} in view of \eqref{4.29}. This then completes the proof of Lemma \ref{NL1} for the second case.

We finally consider the third case where  $\{y_n - x_n\}$ is bounded uniformly in $n>0$. In this case,  the same arguments of  \eqref{4.29} and \eqref{4.31} then yield that  there exists  a function $\widetilde{\phi}\in H^1(\R)$  such that  up to a subsequence  if necessary,
\begin{equation}\label{3.41}
\widetilde{\phi}_n(x+x_n)\to \widetilde{\phi}\ \ \ \text{strongly in}\ C_{loc}^1(\R)\ \text{as}\ n\to\infty,
\end{equation}
and
\begin{eqnarray*}
	-\widetilde {\phi}''+\frac{1}{16}\widetilde {\phi}-w_*^2\widetilde {\phi}=0\ \ \text{in}\,\ \R.
\end{eqnarray*}
Since $\text{ker}\big(-D_{xx}+\frac{1}{16}-w_*^2\big)=\{w_*\}$,  we deduce that
\begin{align}\label{3.42}
\widetilde {\phi}=c_2w_*\ \   \text{holds for some} \,\  c_2\in\R.
\end{align}
Applying \eqref{3.8}, the same argument of \eqref{3.40} thus gives from   \eqref{3.41} and \eqref{3.42} that $\widetilde{\phi}(x)\equiv0$ in  $\R$,
which contradicts again with \eqref{4.29a} for  the third case. This therefore completes the  proof of Lemma \ref{NL1}.
\qed

Together with (\ref{4.10a}) and (\ref{AA3.2}), the following lemma gives the main result of this section, which is concerned with the limiting profile of $(u_{1n}, u_{2n})$ as $ n\to\infty$.

\begin{lem}\label{a}
Suppose $\{a_n\}$ and $\{b_n\}$ are as in  \eqref{3.5}, and let $\{K_n\}$ be defined by \eqref{4.27b}. Then we have
\begin{align}\label{4.43}
\lim\limits_{n\to\infty}a_n&=-\lim\limits_{n\to\infty}b_n=\sqrt{2}/2,\\
 a_n\sqrt{1-b_n^2}(\mu_{2n}-\mu_{1n})&=-b_n\sqrt{1-a_n^2}(\mu_{2n}-\mu_{1n})[1+o(1)]\nonumber\\
&=\frac{1+o(1)}{4}e^{-\sqrt{|\mu_{2n}|}x_n}\ \ \ \text{as}\ \ n\to\infty,\label{4.43b}\\
\sqrt{2}(a_n+b_n)&=[1+o(1)]K_n=-[1+o(1)]\inte w_{1n}w_{2n}dx\nonumber\\
&=-\frac{1+o(1)}{2}x_ne^{-\sqrt{|\mu_{2n}|}\, x_n}\ \ \mbox{as}\ \, n\to\infty, \label{ab1}
\end{align}
and
\begin{eqnarray}\label{phi}
\|\hat\phi_n\|_\infty+\|\hat\psi_n\|_\infty \leq
Ce^{-\sqrt{|\mu_{2n}|}\, x_n}\ \ \ \text{as}\ \ n\to\infty,
\end{eqnarray}
where  $C>0$ is independent of $n>0$.
\end{lem}

\noindent \textbf{Proof.}  We shall carry out the  proof  by three steps.

$Step\ 1$. In this step, we  prove that
\begin{equation}\label{4.45a}
\begin{split}
&a_n\sqrt{1-b_n^2}(\mu_{2n}-\mu_{1n})\\
=&\frac{1}{4}e^{-\sqrt{|\mu_{1n}|}x_n}
+o\Big(e^{-\sqrt{|\mu_{2n}|}x_n}+\big[b_n^2+|a_n|\sqrt{1-b_n^2}\big]|\mu_{2n}-\mu_{1n}|\Big)\ \ \mbox{as}\,\ n\to\infty,
\end{split}
\end{equation}
and
\begin{equation}\label{4.50}
\begin{split}
&-b_n\sqrt{1-a_n^2}(\mu_{2n}-\mu_{1n})\\
=&
\frac{1}{4}e^{-\sqrt{|\mu_{2n}|}x_n}+o\Big(e^{-\sqrt{|\mu_{2n}|}x_n}+\big[b_n^2+|a_n|\sqrt{1-b_n^2}\big]
|\mu_{2n}-\mu_{1n}|\Big)\ \ \mbox{as}\,\ n\to\infty.
\end{split}
\end{equation}

Actually, since
\begin{align}\label{3.51}
&\lim\limits_{n\to\infty}e^{\sqrt{|\mu_{1n}|}x_n}\inte w_{1n}w_{2n}^{2p_n-1}dx\\
=&\inte \lim\limits_{n\to\infty} \Big(e^{\sqrt{|\mu_{1n}|}x_n} w_{1n}(x+x_n)w_{2n}^{2p_n-1}(x+x_n)\Big)dx=1/4,\nonumber
\end{align}
the same argument of \eqref{4.25b} gives that
\begin{align}\label{4.44b}
&\inte \big[(u_{1n}^2+u_{2n}^2)^{p_n-1}-w_{1n}^{2(p_n-1)}\big]w_{1n}w_{2n}dx\nonumber\\
=&\int_{-x_n/2}^{3x_n/2} \big[(u_{1n}^2+u_{2n}^2)^{p_n-1}-w_{1n}^{2(p_n-1)}\big]w_{1n}w_{2n}dx+o\big(e^{-\frac{3}{2}\sqrt{|\mu_{2n}|}x_n})\\
=&\int_{x_n/2}^{3x_n/2} w_{1n}w_{2n}^{2p_n-1}dx+o\big(e^{-\sqrt{|\mu_{2n}|}x_n}+\|\hat\phi_n\|_\infty+\|\hat\psi_n\|_\infty\big)\nonumber\\
=&\frac{1}{4}e^{-\sqrt{|\mu_{1n}|}x_n}+o\big(e^{-\sqrt{|\mu_{2n}|}x_n}+\|\hat \phi_n\|_\infty+\|\hat \psi_n\|_\infty\big)\ \ \ \text{as}\ \ n\to\infty,\nonumber
\end{align}
where we have used \eqref{win} and  the dominated convergence theorem in \eqref{3.51}. 
Moreover, since $\lim\limits_{n\to\infty}\|w_{in}\|_2^2=1$ and $(u_{in}, u_{jn})=\delta_{ij}$ for $i, j=1,2$,
 it can be checked from \eqref{4.6} and \eqref{4.10} that
\begin{align}\label{4.36a}
K_n=&\inte \hat u_{1n}\hat u_{2n}dx\nonumber\\
=& K_n\inte (w_{1n}^2+w_{2n}^2)dx+(1+K_n^2)\inte w_{1n}w_{2n}dx\\
&+\inte\Big[w_{1n}\hat\psi_n+w_{2n}\hat\phi_n+K_n\big(w_{1n}\hat\phi_n+w_{2n}\hat\psi_n\big) +\hat\phi_n\hat\psi_n\Big]dx,\nonumber
\end{align}
which yields  that
\begin{equation}\label{4.36}
\begin{split}
K_n=&O\Big( \inte w_{1n}w_{2n}dx+\|\hat\phi_n\|_\infty+\|\hat\psi_n\|_\infty\Big)\\
=&O\Big( x_ne^{-\sqrt{|\mu_{2n}|}x_n}+\|\hat\phi_n\|_\infty+\|\hat\psi_n\|_\infty\Big)\ \ \ \text{as}\ \ n\to\infty.
\end{split}\end{equation}
Consequently, multiplying both hand sides of \eqref{AA3.5} by $w_{2n}$ and  integrating over $\R$,   one can calculate from \eqref{4.44b} and \eqref{4.36} that
\begin{align}\label{O4}
&L.H.S. \ \ \text{of}\ \inte \eqref{AA3.5} \times w_{2n}dx \nonumber\\
=&\inte \Big\{(-\Delta w_{2n})\hat\phi_n-\mu_{1n}w_{2n}\hat\phi_n-(u_{1n}^2+u_{2n}^2)^{p_n-1}w_{2n}\hat\phi_n\Big\}dx\\
=&\inte \Big\{(\mu_{2n}-\mu_{1n})w_{2n}\hat\phi_n+\Big[w_{2n}^{2(p_n-1)}-({u}_{1n}^2+{u}_{2n}^2)^{p_n-1}\Big]w_{2n}\hat\phi_n\Big\}dx\nonumber\\[1.5mm]
=&o\big(\|\hat\phi_n\|_\infty\big)\ \ \ \text{as}\ \ n\to\infty,\nonumber
\end{align}
and
\begin{align}\label{O5}
&R.H.S. \ \ \text{of}\ \inte \eqref{AA3.5} \times w_{2n}dx \nonumber\\
=&\inte \Big[(u_{1n}^2+u_{2n}^2)^{p_n-1}-w_{1n}^{2(p_n-1)}\Big]w_{1n}w_{2n}dx\nonumber\\
&+K_n\inte\Big[(u_{1n}^2+u_{2n}^2)^{p_n-1}-w_{2n}^{2(p_n-1)}\Big]w^2_{2n}dx+a_n\sqrt{1-b_n^2}(\mu_{1n}-\mu_{2n})\|w_{2n}\|_2^2\nonumber\\
&
+b_n^2(\mu_{1n}-\mu_{2n})\inte w_{1n}w_{2n}dx +o\big(\|\hat\phi_n\|_\infty+\|\hat\psi_n\|_\infty\big)\\
=&\frac{1}{4}e^{-\sqrt{|\mu_{1n}|}x_n}+a_n\sqrt{1-b_n^2}(\mu_{1n}-\mu_{2n})\nonumber\\
&+o\Big(e^{-\sqrt{|\mu_{2n}|}x_n} +(b_n^2+|a_n|\sqrt{1-b_n^2})|\mu_{2n}-\mu_{1n}|+\|\hat\phi_n\|_\infty+\|\hat\psi_n\|_\infty\Big)\ \ \ \text{as}\ \ n\to\infty.\nonumber
\end{align}
Applying Lemma \ref{NL1},  we immediately obtain  that  \eqref{4.45a} holds true.
Similarly, multiplying both hand sides of \eqref{AA3.6} by $w_{1n}$ and  integrating over $\R$,  one can deduce that \eqref{4.50} holds true, and Step 1 is thus done.

$Step\ 2$.
In this step, we prove that  \eqref{4.43b} and \eqref{phi} hold true. We first claim that
\begin{align}\label{01}
	\lim\limits_{n\to\infty}a_n\neq0.
\end{align}
On the contrary, suppose  $\lim\limits_{n\to\infty}a_n=0.$
It then yields from \eqref{3.5} and \eqref{4.50} that $-\lim\limits_{n\to\infty}b_n=\lim\limits_{n\to\infty}a_n=0$ and
\begin{align}\label{ab}
\lim\limits_{n\to\infty}(\mu_{2n}-\mu_{1n})e^{\sqrt{|\mu_{2n}|}x_n}=\infty.
\end{align}
Using the fact that $\inte \big(\hat{u}_{1n}^2-\hat{u}_{2n}^2\big)dx=0$, we
have
\begin{align}\label{4.42}
0=&\inte(1-K_n^2)(w_{1n}^2-w_{2n}^2)dx\\
&+\inte \Big[2(w_{1n}+K_nw_{2n})\hat\phi_n-2(K_nw_{1n}+w_{2n})\hat\psi_n\Big]dx+\|\hat\phi_n\|_2^2-\|\hat\psi_n\|_2^2.\nonumber
\end{align}
We also calculate from \eqref{win} that
\begin{align}\label{4.44a}
	\inte(w_{1n}^2-w_{2n}^2)dx=  \big[8+o(1)\big](\mu_{2n}-\mu_{1n})\ \ \ \text{as}\ \ n\to\infty.
\end{align}
We thus deduce from \eqref{4.42}, \eqref{4.44a} and Lemma \ref{NL1} that
\begin{align}\label{4.45}
&|\mu_{2n}-\mu_{1n}|
\leq
C\Big[e^{-\sqrt{|\mu_{2n}|}\, x_n}+\big(b_n^2+|a_n|\sqrt{1-b_n^2}\big)|\mu_{2n}-\mu_{1n}|\Big]\ \ \ \text{as}\ \ n\to\infty.
\end{align}
Since $-\lim\limits_{n\to\infty}b_n=\lim\limits_{n\to\infty}a_n=0$, we  derive  from \eqref{4.45} that $|\mu_{2n}-\mu_{1n}|
\leq Ce^{-\sqrt{|\mu_{2n}|}\, x_n}$ uniformly as  $n\to\infty$, which however contradicts with \eqref{ab}. Hence,  the claim \eqref{01} is proved.

By the fact that $\lim\limits_{n\to\infty}a_n=-\lim\limits_{n\to\infty}b_n\in[0, \sqrt{2}/2]$, we deduce from \eqref{4.50}  and \eqref{01} that for some $C>0,$
\begin{align*}
-b_n\sqrt{1-a_n^2}(\mu_{2n}-\mu_{1n})
=\frac{1+o(1)}{4}e^{-\sqrt{|\mu_{2n}|}x_n}\ \ \ \text{as}\ \ n\to\infty,
\end{align*}
and
\begin{align}\label{4.63}
|\mu_{2n}-\mu_{1n}|=\big[C+o(1)\big]e^{-\sqrt{|\mu_{2n}|}x_n}\ \ \text{as}\ \ n\to\infty.
\end{align}
Employing Lemma \ref{NL1}, we hence conclude from \eqref{4.45a} that \eqref{4.43b} and \eqref{phi} hold true. This proves Step 2.

$Step\ 3$.  Denote
$$\mathcal{L}_{in}:=-D_{xx}-\mu_{in}-(2p_n-1)w_{in}^{2(p_n-1)}\ \ \ \text{in}\ \ \R, \ \ i=1,2.$$
Since $$w_{1n}=\frac{1}{2\mu_{1n}}\mathcal{L}_n\Big(xw_{1n}'+\frac{1}{p_n-1}w_{1n}\Big),$$
and
$$w_{2n}=\frac{1}{2\mu_{2n}}\mathcal{L}_n\Big((x-x_n)w_{2n}'+\frac{1}{p_n-1}w_{2n}\Big),$$
we have
\begin{align}\label{4.57}
	\inte w_{1n}\hat \phi_ndx
	=&\frac{1}{2\mu_{1n}}\inte \mathcal{L}_{1n}\hat\phi_n\Big(xw_{1n}'+\frac{1}{p_n-1}w_{1n}\Big)dx,
\end{align}
and
\begin{align}\label{4.58}
	\inte w_{2n}\hat \psi_ndx
	=&\frac{1}{2\mu_{2n}}\inte \mathcal{L}_{2n}\hat\psi_n\Big[(x-x_n)w_{2n}'+\frac{1}{p_n-1}w_{2n}\Big]dx.
\end{align}
One can calculate from \eqref{AA3.5} and \eqref{4.36} that  for some $\delta\in(0, 1)$,
\begin{align*}
	&\inte \mathcal{L}_{1n}\hat\phi_n\Big[xw_{1n}'+\frac{1}{p_n-1}w_{1n}\Big]dx\\
	=&\inte \big[(u_{1n}^2+u_{2n}^2)^{p_n-1}-w_{1n}^{2(p_n-1)}\big]\hat\phi_n\Big(xw_{1n}'+\frac{1}{p_n-1}w_{1n}\Big)dx\\
	&+\inte  F_{1n}(x)\Big(xw_{1n}'+\frac{1}{p_n-1}w_{1n}\Big)dx\\
	&+K_n\inte \big[(u_{1n}^2+u_{2n}^2)^{p_n-1}-w_{2n}^{2(p_n-1)}\big]w_{2n}\Big(xw_{1n}'+\frac{1}{p_n-1}w_{1n}\Big)dx\\
	&-b_n^2(\mu_{1n}-\mu_{2n})\inte w_{1n}\Big(xw_{1n}'+\frac{1}{p_n-1}w_{1n}\Big)dx
	+o(|\mu_{1n}-\mu_{2n}|)\\
	 =&-b_n^2(\mu_{1n}-\mu_{2n})\frac{3-p_n}{2(p_n-1)}\|w_{1n}\|_2^2+o\Big(e^{-(1+\delta)\sqrt{|\mu_{2n}|}x_n}+|\mu_{1n}-\mu_{2n}|+\|\hat\phi_n\|_\infty\Big)
\end{align*}
as $n\to\infty$, where the function $F_{1n}$ is as in \eqref{O1}. This then yields from  \eqref{4.57} that
\begin{align}\label{4.59}
	&\inte w_{1n}\hat \phi_ndx=
	4b_n^2(\mu_{1n}-\mu_{2n})+o\Big(e^{-\sqrt{|\mu_{2n}|}x_n}+|\mu_{1n}-\mu_{2n}|\Big)\ \ \ \text{as}\ \ n\to\infty.
\end{align}
Similarly, we have
\begin{align}\label{4.60}
	&\inte w_{2n}\hat \psi_ndx=
	-4a_n^2(\mu_{1n}-\mu_{2n})+o\Big(e^{-\sqrt{|\mu_{2n}|}x_n}+|\mu_{1n}-\mu_{2n}|\Big)\ \ \ \text{as}\ \ n\to\infty.
\end{align}

Moreover, one can check from \eqref{phi}, \eqref{4.42},  \eqref{4.44a} \eqref{4.59} and \eqref{4.60} that
\begin{align*}
8(\mu_{2n}-\mu_{1n})=&-2\inte \big(w_{1n}\hat\phi_n-w_{2n}\hat\psi_n\big)dx+o(\mu_{2n}-\mu_{1n})\\
=&16a_n^2(\mu_{2n}-\mu_{1n})+o(\mu_{2n}-\mu_{1n}),
\end{align*}
which  implies that $\lim\limits_{n\to\infty}b_n^2=\lim\limits_{n\to\infty}a_n^2=1/2$. Note from \eqref{3.5} that  $-\lim\limits_{n\to\infty}b_n=\lim\limits_{n\to\infty}a_n\in(0, \sqrt{2}/2]$. We thus deduce that
$-\lim\limits_{n\to\infty}b_n=\lim\limits_{n\to\infty}a_n=\sqrt{2}/2$, which further yields that
\begin{align*}
	K_n=&a_n\sqrt{1-b_n^2}+b_n\sqrt{1-a_n^2}\\
	=&(a_n+b_n)\Big[\sqrt{1-b_n^2}+b_n\frac{b_n-a_n}{\sqrt{1-a_n^2}+\sqrt{1-b_n^2}}\Big]\\
	=&\sqrt{2}(a_n+b_n)[1+o(1)]\ \ \ \text{as}\ \ n\to\infty.
\end{align*}
Finally, we derive  from \eqref{win}, \eqref{4.36a} and \eqref{4.63}  that
\begin{align*}
K_n=&-[1+o(1)]\inte w_{1n}w_{2n}dx=-\frac{1+o(1)}{2}\int_0^{x_n}e^{-\sqrt{|\mu_{2n}|}x}e^{-\sqrt{|\mu_{2n}|}(x_n-x)}dx\\
=&-\frac{1+o(1)}{2} x_ne^{-\sqrt{|\mu_{2n}|}x_n}\ \ \ \text{as}\ \ n\to\infty.
\end{align*}
This therefore completes the proof of Lemma \ref{a}. \qed

\section{Refined Estimates of $x_n$ as $n\to\infty$}
In this section we first analyze the refined estimate of the global maximal point $x_n>0$ defined in (\ref{1.6K}), based on which we then complete in Subsection 4.1 the proofs of Theorem  \ref{thm1.1} and Proposition \ref{prop1.2}.

Applying Lemma \ref{a}, we begin with the following estimates of $x_n$ as $p_n \to 2$.

\begin{lem}\label{N24}
Let the point $x_n>0$ be as in \eqref{3.2}. Then we have
$$\lim\limits_{n\to\infty}(2-p_n)x_n=0.$$
\end{lem}

\noindent \textbf{Proof.}  Let the functions $h_{1n}$ and  $h_{2n}$ be as in \eqref{4.35} and \eqref{4.35a}, respectively. Using Taylor's expansion, one can calculate from \eqref{4.27a}  and \eqref{phi} that  there exist $\theta_{1n}, \theta_{2n}\in(0, 1)$ and a sufficiently small constant $\varepsilon\in(0, 1/2)$ such that
\begin{align*}
&\int_{x_n/2}^{3x_n/2}w_{1n}w_{1n}'(u_{1n}^2+u_{2n}^2)^{p_n-1}dx
=\int_{x_n/2}^{3x_n/2}w_{1n}w_{1n}'(w_{2n}^2+h_{1n})^{p_n-1}dx\\
=&\int_{x_n/2}^{3x_n/2}w_{1n}w_{1n}'\Big\{w_{2n}^{2p_n-2}+(p_n-1)\big[w_{2n}^2+\theta_{1n}h_{1n}\big]^{p_n-2}h_{1n}\Big\}dx\\
=&\int_{x_n/2}^{3x_n/2}w_{1n}w_{1n}'\Big\{w_{2n}^{2p_n-2}+(p_n-1)\big[1+\theta_{1n}h_{1n}/w_{2n}^2\big]^{p_n-2}
w_{2n}^{2p_n-4}h_{1n}\Big\}dx\\
=&\int_{x_n/2}^{3x_n/2}w_{1n}w_{1n}'\Big\{w_{2n}^{2p_n-2}+\big[1+O(p_n-2)\big]w_{2n}^{2p_n-4}w_{1n}^2\Big\}dx+o\big(e^{-(2+\varepsilon)\sqrt{|\mu_{2n}|}x_n}\big)
\end{align*}
as $n\to\infty$, and
\begin{align*}
&\int_{-3x_n/4}^{x_n/2} w_{1n}w_{1n}'\big[(u_{1n}^2+u_{2n}^2)^{p_n-1}-w_{1n}^{2p_n-2}\big]dx\\
=&(p_n-1)\int_{-3x_n/4}^{x_n/2}w^{2p_n-3}_{1n}w_{1n}'h_{2n}dx\\
&+(p_n-1)(p_n-2)\int_{-3x_n/4}^{x_n/2}(w_{1n}^2+\theta_{2n}h_{2n})^{p_n-3}w_{1n}w_{1n}'h_{2n}^2dx\\
=&(p_n-1)\int_{-3x_n/4}^{x_n/2}w^{2p_n-3}_{1n}w_{1n}'h_{2n}dx+O\big((p_n-2)e^{-2\sqrt{|\mu_{2n}|}x_n}\big)
\end{align*}
as $n\to\infty$.
These then yield  that for the above sufficiently  small $\varepsilon\in(0, 1/2)$,
\begin{align}\label{4.70b}
&\inte \Big\{w_{1n}w_{1n}'\big[(u_{1n}^2+u_{2n}^2)^{p_n-1}-w_{1n}^{2p_n-2}\big]-2(p_n-1) w_{1n}^{2p_n-2}w_{1n}'\hat\phi_n\Big\}dx\nonumber\\
=&\int_{-3x_n/4}^{3x_n/2} \Big\{w_{1n}w_{1n}'\big[(u_{1n}^2+u_{2n}^2)^{p_n-1}-w_{1n}^{2p_n-2}\big]-2(p_n-1) w_{1n}^{2p_n-2}w_{1n}'\hat\phi_n\Big\}dx\nonumber\\
&+o\big(e^{-(2+\varepsilon)\sqrt{|\mu_{2n}|}x_n}\big)\nonumber\\
=&\int_{x_n/2}^{3x_n/2}w_{1n}w_{1n}'\Big\{w_{2n}^{2p_n-2}+\big[1+O(p_n-2)\big]w_{2n}^{2p_n-4}w_{1n}^2\Big\}dx-\int_{x_n/2}^{3x_n/2}w_{1n}^{2p_n-1}w_{1n}'dx\nonumber\\
&+(p_n-1)\int_{-3x_n/4}^{x_n/2}w^{2p_n-3}_{1n}w_{1n}'h_{2n}dx-2(p_n-1)\int_{-3x_n/4}^{x_n/2}w_{1n}^{2p_n-2}w_{1n}'\hat\phi_ndx\\
&+O\big((p_n-2)e^{-2\sqrt{|\mu_{2n}|}x_n}\big)\nonumber\\
=&\int_{x_n/2}^{3x_n/2}w_{1n}w_{1n}'w_{2n}^{2p_n-2}dx
+(p_n-1)\int_{-3x_n/4}^{x_n/2}w^{2p_n-3}_{1n}w_{1n}'\big[w_{2n}^2+2w_{2n}\hat\psi_n+2K_nw_{1n}w_{2n}\big]dx\nonumber\\
&+O\big(e^{-p_n\sqrt{|\mu_{2n}|}x_n}+e^{-2\sqrt{|\mu_{2n}|}x_n}\big)\ \ \ \text{as}\ \ n\to\infty.\nonumber
\end{align}
Similarly, we derive that
\begin{align}\label{4.70c}
&\inte w_{1n}'w_{2n}\big[(u_{1n}^2+u_{2n}^2)^{p_n-1}-w_{2n}^{2p_n-2}\big]dx\nonumber\\
=&\int_{-x_n/2}^{x_n/2}w_{1n}'w_{2n}\Big\{w_{1n}^{2p_n-2}+\big[1+o(1)\big](p_n-1)w_{1n}^{2p_n-4}h_{2n}\Big\}dx\\
&+[1+o(1)](p_n-1)\int_{x_n/2}^{3x_n/2}w_{1n}'w_{2n}^{2p_n-3}h_{1n}dx+o\big(e^{-\frac{5}{4}\sqrt{|\mu_{2n}|}x_n}\big)\nonumber\\
=&\int_{-x_n/2}^{x_n/2}w_{1n}'w_{2n}w_{1n}^{2p_n-2}dx+o\big(e^{-\frac{5}{4}\sqrt{|\mu_{2n}|}x_n}\big)\ \ \ \text{as}\ \ n\to\infty.\nonumber
\end{align}
Hence, multiplying both hand sides of \eqref{AA3.5} by $w'_{1n}$ and integrating over $\R$, we  conclude from \eqref{AA3.1}, \eqref{phi}, \eqref{4.70b} and \eqref{4.70c} that  
\begin{align*}
&\inte LHS\ \text{of}\  \eqref{AA3.5} \times w'_{1n}dx-2(p_n-1)\inte  w_{1n}^{2p_n-2}w_{1n}'\hat \phi_ndx\\
=&\inte\big[w_{1n}^{2p_n-2}- (u_{1n}^2+u_{2n}^2)^{p_n-1}\big]w_{1n}'\hat \phi_ndx
=O\big(e^{-2\sqrt{|\mu_{2n}|}x_n}\big)\ \ \ \text{as}\ \ n\to\infty,
\end{align*}
and
\begin{align}\label{N4}
&\inte RHS \ \text{of}\ \eqref{AA3.5} \times w'_{1n}dx-2(p_n-1)\inte  w_{1n}^{2p_n-2}w_{1n}'\hat \phi_ndx\nonumber\\
=&\inte w_{1n}w_{1n}'\big[(u_{1n}^2+u_{2n}^2)^{p_n-1}-w_{1n}^{2p_n-2}\big]dx-2(p_n-1)\inte  w_{1n}^{2p_n-2}w_{1n}'\hat \phi_ndx\nonumber\\
&+\inte K_nw_{1n}'w_{2n}\big[(u_{1n}^2+u_{2n}^2)^{p_n-1}-w_{2n}^{2p_n-2}\big]dx\nonumber\\
&+a_n\sqrt{1-b_n^2}(\mu_{1n}-\mu_{2n})\inte w_{2n}w_{1n}'dx -b_n(\mu_{1n}-\mu_{2n})\inte w_{1n}'\psi_ndx\\
=&\int_{x_n/2}^{3x_n/2}w_{1n}w_{1n}'w_{2n}^{2p_n-2}dx
+(p_n-1)\int_{-3x_n/4}^{x_n/2}w^{2p_n-3}_{1n}w_{1n}'w_{2n}^2dx\nonumber\\
&+(p_n-1)2K_n\int_{-3x_n/4}^{x_n/2}w_{1n}'w_{2n}w^{2p_n-2}_{1n}dx+K_n\int_{-x_n/2}^{x_n/2}w_{1n}'w_{2n}w_{1n}^{2p_n-2}dx\nonumber\\
&+a_n\sqrt{1-b_n^2}(\mu_{1n}-\mu_{2n})\inte w_{2n}w_{1n}'dx +O\big(e^{-p_n\sqrt{|\mu_{2n}|}x_n}+e^{-2\sqrt{|\mu_{2n}|}x_n}\big)\ \ \ \text{as}\ \ n\to\infty,\nonumber
\end{align}
which thus  imply that
\begin{align}\label{N7}
&O\Big(e^{-p_n\sqrt{|\mu_{2n}|}x_n}+e^{-2\sqrt{|\mu_{2n}|}x_n}\Big) \nonumber\\
=&\int_{x_n/2}^{3x_n/2}w_{1n}w_{1n}'w_{2n}^{2p_n-2}dx
+(p_n-1)\int_{-3x_n/4}^{x_n/2}w^{2p_n-3}_{1n}w_{1n}'w_{2n}^2dx\\
&+(p_n-1)2K_n\int_{-3x_n/4}^{x_n/2}w_{1n}'w_{2n}w^{2p_n-2}_{1n}dx+K_n\int_{-x_n/2}^{x_n/2}w_{1n}'w_{2n}w_{1n}^{2p_n-2}dx\nonumber\\
&+a_n\sqrt{1-b_n^2}(\mu_{1n}-\mu_{2n})\inte w_{2n}w_{1n}'dx \nonumber\\
:=&{A}_n+{B}_n+{C}_n\ \ \ \text{as}\ \ n\to\infty.\nonumber
\end{align}
We next estimate ${A}_n, {B}_n$ and ${C}_n$ as $n\to\infty$.

By direct calculations, it yields from \eqref{win} and \eqref{4.63}  that
\begin{align}\label{4.74a}
A_n=&\int_{x_n/2}^{x_n}w_{1n}w_{1n}'w_{2n}^{2p_n-2}dx
+(p_n-1)\int_{0}^{x_n/2}w^{2p_n-3}_{1n}w_{1n}'w_{2n}^2dx+O(e^{-2\sqrt{|\mu_{2n}|}x_n})\nonumber\\
=&-\frac{1+o(1)}{16}\int_{x_n/2}^{x_n}e^{-2\sqrt{|\mu_{2n}|}x}e^{-(2p_n-2)\sqrt{|\mu_{2n}|}(x_n-x)}dx\nonumber\\
&-\frac{1+o(1)}{16}\int_0^{x_n/2}e^{-(2p_n-2)\sqrt{|\mu_{2n}|}x}e^{-2\sqrt{|\mu_{2n}|}(x_n-x)}dx+O(e^{-2\sqrt{|\mu_{2n}|}x_n})\\
=&-\frac{1+o(1)}{16(p_n-2)\sqrt{|\mu_{2n}|}}\Big(e^{-2\sqrt{|\mu_{2n}|}x_n}-e^{-p_n\sqrt{|\mu_{2n}|}x_n}\Big)\ \ \ \text{as}\ \ n\to\infty,\nonumber
\end{align}
and
\begin{equation}\label{4.6a}
\begin{split}
&\lim\limits_{n\to\infty}\int_{-3x_n/4}^{x_n/2} e^{\sqrt{|\mu_{2n}|}x_n}w_{1n}^{2p_n-2}w_{1n}'w_{2n}dx\\
=&\inte \lim\limits_{n\to\infty}e^{\sqrt{|\mu_{2n}|}x_n}w_{1n}^{2p_n-2}w_{1n}'w_{2n}dx=-1/48,
\end{split}
\end{equation}
and
\begin{equation}\label{4.7}
\begin{split}
\inte w_{2n}w_{1n}'=&-\frac{1+o(1)}{8}\int_0^{x_n}e^{-\sqrt{|\mu_{2n}|}(x_n-x)}e^{-\sqrt{|\mu_{2n}|}x}dx+O\big(e^{-\sqrt{|\mu_{2n}|}x_n}\big)\nonumber\\
=&-\frac{1+o(1)}{8}x_ne^{-\sqrt{|\mu_{2n}|}x_n}\ \ \ \text{as}\ \ n\to\infty.
\end{split}
\end{equation}
Applying Lemma \ref{a}, we then deduce from \eqref{4.6a} and \eqref{4.7} that
\begin{eqnarray}\label{4.70a}
\left\{
\begin{array}{lll}
	\!	\!	\!
{B}_n=\displaystyle\frac{1+o(1)}{32}x_ne^{-2\sqrt{|\mu_{2n}|}x_n}\ \ \ \text{as}\ \ n\to\infty,\\[2mm]
	\!	\!	\!
{C}_n=\displaystyle\frac{1+o(1)}{32}x_ne^{-2\sqrt{|\mu_{2n}|}x_n}\ \ \ \text{as}\ \ n\to\infty,
\end{array}	
\right.
\end{eqnarray}
which further yields  from \eqref{N7} and \eqref{4.74a} that
\begin{align*}
\left\{
\begin{array}{lll}
	\!	\!	\!
	&\!	\!	\!\!	\!	\! \displaystyle\frac{1}{(p_n-2)\sqrt{|\mu_{2n}|}}\Big(1-e^{(2-p_n)\sqrt{|\mu_{2n}|}x_n}\Big)-x_n=O(1)
	\ \ \ \text{as}\ \  p_n\searrow2,\\[3.5mm]
	\!	\!	\! &\!	\!	\!\!	\!	\! \displaystyle\frac{1}{(p_n-2)\sqrt{|\mu_{2n}|}}\Big(e^{(p_n-2)\sqrt{|\mu_{2n}|}x_n}-1\Big) -x_ne^{(p_n-2)\sqrt{|\mu_{2n}|}x_n}=O(1)
	\ \ \ \text{as}\ \  p_n\nearrow2,
\end{array}	
\right.
\end{align*}
and thus
\begin{align*}
	\left\{
	\begin{array}{lll}
\!	\!	\!
&\!	\!	\!\!	\!	\! e^{(2-p_n)\sqrt{|\mu_{2n}|}x_n}-1-(2-p_n)\sqrt{|\mu_{2n}|}x_n=o(1)
\ \ \ \text{as}\ \  p_n\searrow2,\\[3mm]
\!	\!	\! &\!	\!	\!\!	\!	\! e^{(p_n-2)\sqrt{|\mu_{2n}|}x_n}-1-(p_n-2)\sqrt{|\mu_{2n}|}x_ne^{(p_n-2)\sqrt{|\mu_{2n}|}x_n}=o(1)
\ \ \ \text{as}\ \  p_n\nearrow2.
\end{array}	
\right.
\end{align*}
Since both equations $e^t-1-t=0$ and $e^t-1-te^t=0$ have  exactly  one solution $t=0$ in $(-\infty, 0]$, we derive that $\lim\limits_{n\to\infty}(p_n-2)\sqrt{|\mu_{2n}|}x_n=0$, which hence proves Lemma \ref{N24}. \qed

Following Lemma \ref{N24}, we are next concerned with the following refined estimate of $x_n$ as $n\to\infty$.

\begin{lem}\label{lem4.7}
Under the assumptions of Lemma \ref{N24}, the point $x_n>0$ satisfies
$$\lim\limits_{n\to\infty}(2-p_n)x_n^2=48.$$
\end{lem}

\noindent \textbf{Proof.} We first claim that  for sufficiently large $n>0$,
\begin{align}\label{4.69a}
\left\{
\begin{array}{lll}
\!	\!	\!
\hat\psi_n(x)=-\frac{x}{2}w_*(x) e^{-\sqrt{|\mu_{2n}|}x_n}+o(e^{-\sqrt{|\mu_{2n}|}x_n})\ \ \ \text{in}\ B_{x_n/4}(0),\\[2mm]
\!	\!	\! \hat\phi_n(x)=-2(xw_*'+w_*) e^{-\sqrt{|\mu_{2n}|}x_n}+o(e^{-\sqrt{|\mu_{2n}|}x_n})\ \ \ \text{in}\,\ B_{x_n/4}(0),
\end{array}	
\right.
\end{align}
and
\begin{align}\label{4.73}
\left\{
\begin{array}{lll}
\!	\!	\!
\hat\psi_n(x)=&\!\!\!2\big[(x-x_n)w_*'(x-x_n)+w_*(x-x_n)\big] e^{-\sqrt{|\mu_{2n}|}x_n}&\\
&+o(e^{-\sqrt{|\mu_{2n}|}x_n})
\ \ \ \text{in}\ B_{x_n/4}(x_n),&\\[2mm]
\!	\!	\!  \hat\phi_n(x)=&\!\!\!\frac{x-x_n}{2}w_*(x-x_n)e^{-\sqrt{|\mu_{2n}|}x_n}+o(e^{-\sqrt{|\mu_{2n}|}x_n})\ \ \ \text{in}\ \, B_{x_n/4}(x_n),&
\end{array}	
\right.
\end{align}
where  $\hat\phi_n$ and $\hat \psi_n$ are given by \eqref{AA3.2}, and $w_*>0$ is as in  \eqref{2.8b}.

Actually, 
one can derive from \eqref{phi} that  for sufficiently large $n>0$,
\begin{align}\label{4.69}
\hat\psi_n(x)=\hat\psi_0(x) e^{-\sqrt{|\mu_{2n}|}x_n}+o\big(e^{-\sqrt{|\mu_{2n}|}x_n}\big)\ \ \ \text{in}\,\ B_{x_n/4}(0),
\end{align}
and
\begin{align}\label{4.69b}
\hat\phi_n(x)=\hat\phi_0(x) e^{-\sqrt{|\mu_{2n}|}x_n}+o\big(e^{-\sqrt{|\mu_{2n}|}x_n}\big)\ \ \ \text{in}\,\ B_{x_n/4}(0)
\end{align}
hold for some functions $\hat\psi_0,\ \hat\phi_0\in L^\infty(\R)$. 
Since $\lim\limits_{n\to\infty}(2-p_n)x_n=0$ and $\mu_{2n}-\mu_{1n}=\frac{1+o(1)}{2}e^{-\sqrt{|\mu_{2n}|}x_n}$ as $n\to\infty$,  one can calculate  that
\begin{align}\label{4.70}
&w_{1n}^{2p_n-3}(x)w_{2n}(x)\nonumber\\
=&\frac{1+o(1)}{2e^{\sqrt{|\mu_{2n}|}x_n}}\frac{1}{1+e^{2\sqrt{|\mu_{2n}|}(x-x_n)}+e^{-2\sqrt{|\mu_{2n}|}x}+e^{-2\sqrt{|\mu_{2n}|}x_n}}\\
=&\frac{1+o(1)}{2e^{\sqrt{|\mu_{2n}|}x_n}}\frac{1}{1+e^{-2\sqrt{|\mu_{2n}|}x}}
=\frac{1+o(1)}{2e^{\sqrt{|\mu_{2n}|}x_n}}\frac{e^{\sqrt{|\mu_{2n}|}x}}{e^{\sqrt{|\mu_{2n}|}x}+e^{-\sqrt{|\mu_{2n}|}x}}\nonumber
\end{align}
uniformly in $B_{x_n/4}(0)$ as $n\to\infty$. Applying \eqref{AA3.6} and Lemma \ref{a}, we thus conclude from \eqref{4.69} and \eqref{4.70} that
\begin{align*}
-\hat\psi_0''+\frac{1}{16}\hat\psi_0-w_*^2\hat\psi_0
=&e^{\sqrt{|\mu_{2n}|}x_n}w_{1n}\Big\{w_{2n}w_{1n}^{2p_n-3}-b_n\sqrt{1-a_n^2}(\mu_{1n}-\mu_{2n})\Big\}+o(1)\\
=&\frac{1}{4}e^{\sqrt{|\mu_{2n}|}x_n}w_{1n}\Big[4w_{2n}w_{1n}^{2p_n-3}-e^{-\sqrt{|\mu_{2n}|}x_n}\Big]+o(1)\\
=&\frac{1+o(1)}{4}w_{1n}\frac{e^{\sqrt{|\mu_{2n}|}x}-e^{-\sqrt{|\mu_{2n}|}x}}{e^{\sqrt{|\mu_{2n}|}x}+e^{-\sqrt{|\mu_{2n}|}x}}+o(1)\\
=&w_{1n}'+o(1)\ \ \text{in}\ B_{x_n/4}(0)\  \ \text{as}\ \, n\to\infty,
\end{align*}
which shows that
\begin{align}\label{4.71}
\mathcal{L}_1\psi_0:=	\Big(-D_{xx}+\frac{1}{16}-w_*^2\Big)\hat\psi_0
=w_*'\ \ \text{in}\,\ \R.
\end{align}
Note that  $$\text{ker}\mathcal{L}_1=\{w_*\},\ \  \ \ \mathcal{L}_1\Big(-\frac{x}{2}w_*\Big)=w_*' \ \ \text{in} \ \, \R.$$
It then follows from \eqref{4.71} that
$\hat\psi_0=-\frac{x}{2}w_*(x)+cw_*(x)$ in $\R$ for some $c\in\R$. It thus yields from \eqref{3.5}, \eqref{3.7} and  \eqref{AA3.2} that
\begin{align*}
0=&\lim\limits_{n\to\infty}e^{\sqrt{|\mu_{2n}|}x_n}\inte \big[-b_n\phi_n+\sqrt{1-b_n^2}\psi_n\big]w_{1n}dx
=\lim\limits_{n\to\infty}e^{\sqrt{|\mu_{2n}|}x_n}\inte \hat \psi_nw_{1n}dx\\
=& \inte \big[-\frac{x}{2}w_*(x)+cw_*(x)\big]w_*dx=c\inte w_*^2dx,
\end{align*}
which further yields that $c=0$. We therefore obtain that
\begin{align}\label{4.72}
\hat\psi_n(x)=-\frac{x}{2}w_*(x) e^{-\sqrt{|\mu_{2n}|}x_n}+o\big(e^{-\sqrt{|\mu_{2n}|}x_n}\big)\ \ \ \text{in}\ B_{x_n/4}(0).
\end{align}
Similarly, one can check from \eqref{AA3.5} and \eqref{4.69b}  that
\begin{align*}
\mathcal{L}_2\hat\phi_0:=-\hat\phi_0''+\frac{1}{16}\hat\phi_0-3w_*^2\hat\phi_0
=&\frac{1}{4}w_*\ \ \text{in}\ \, \R.
\end{align*}
Since ker$\mathcal{L}_2=\{w_*'\}$ and $w_*=-8\mathcal{L}_2\big(xw_*'+w_*\big)$,  we obtain from above that
\begin{align}\label{4.15}
\hat\phi_0=-2(xw_*'+w_*)+cw_*'\ \ \ \text{holds for some}\ \ c\in\R.
\end{align}
Since it yields from \eqref{3.7} and \eqref{AA3.2}
that $\inte \hat \phi_nw'_{1n}dx=0 $, we get from \eqref{4.15} that  $\hat\phi_0(x)=-2\big[xw_*'(x)+w_*(x)\big]$ in $\R$, which therefore implies that the claim \eqref{4.69a} holds true.
Similar to  \eqref{4.69a},  one can deduce from \eqref{3.8} and \eqref{AA3.2} that the claim \eqref{4.73} holds true, and we are done.

Recalling from Lemma \ref{N24} that $\lim\limits_{n\to\infty}(2-p_n)x_n=0$, we have
\begin{align*}
&\int_{x_n/2}^{3x_n/2}\Big[w_{1n}^3w_{1n}'w_{2n}^{2p_n-4}-w_{1n}^{2p_n-1}w_{1n}'\Big]dx\\
=&\frac{1+o(1)}{16}\int_{x_n/2}^{3x_n/2}\Big[ e^{-2p_n\sqrt{|\mu_{2n}|}x}-e^{-4\sqrt{|\mu_{2n}|}x}\Big]dx=o\big(e^{-2\sqrt{|\mu_{2n}|}x_n}\big)\ \ \mbox{as}\,\ n\to\infty.
\end{align*}
Multiplying both hand sides of \eqref{AA3.5} by $w'_{1n}$ and integrating over $\R$,  the same calculations of \eqref{4.70b} and \eqref{N7} then show  that
\begin{align}\label{4.74}
&o\big(e^{-2\sqrt{|\mu_{2n}|}x_n}\big) \nonumber\\
=&A_n+B_n+C_n+\inte\big[(u_{1n}^2+u_{2n}^2)^{p_n-1}-w_{1n}^{2p_n-2}\big]w_{1n}'\hat \phi_ndx\nonumber\\
&+2(p_n-1)\int_{-3x_n/4}^{x_n/2} w_{1n}'w_{1n}^{2p_n-3}w_{2n}\hat\psi_ndx-b_n(\mu_{1n}-\mu_{2n})\inte w_{1n}'\psi_ndx\\
&+(p_n-1)\int_{-3x_n/4}^{x_n/2} w_{1n}'w_{1n}^{2p_n-3}(\phi_n^2+\psi_n^2)dx\nonumber\\
:=&A_n+B_n+C_n+D_n+E_n+F_n+G_n\ \ \ \text{as}\ \ n\to\infty,\nonumber
\end{align}
where  $A_n, B_n$ and  $C_n$ are as in \eqref{N7}. We calculate from \eqref{4.69a} and \eqref{4.73} that for some sufficiently small  $\varepsilon>0$,
\begin{align}\label{4.75}
D_n
=&(p_n-1)\int_{B_{x_n/2}(0)}w_{1n}^{2p_n-4}\big[w_{2n}^2+2w_{1n}\hat\phi_n\big]w_{1n}'\hat \phi_ndx\nonumber\\
&+\int_{B_{x_n/2}(x_n)}w_{2n}^{2p_n-2}w_{1n}'\hat \phi_ndx+o\big(e^{-(2+\varepsilon)\sqrt{|\mu_{2n}|}x_n}\big)\nonumber\\
=&
\int_{B_{x_n/4}(x_n)}w_{2n}^{2p_n-2}w_{1n}'\hat \phi_ndx+o\big(e^{-(2+\varepsilon)\sqrt{|\mu_{2n}|}x_n}\big)\\
=&\frac{e^{-\sqrt{|\mu_{2n}|}x_n}}{2}\inte (x-x_n)\cdotp w_*^3(x-x_n)w_*'(x)dx+o\big(e^{-2\sqrt{|\mu_{2n}|}x_n}\big)\nonumber\\
=&\frac{1}{16}e^{-2\sqrt{|\mu_{2n}|}x_n}+o\big(e^{-2\sqrt{|\mu_{2n}|}x_n}\big)
\ \ \ \text{as}\ \ n\to\infty,\nonumber\\
E_n=&2(p_n-1)\int_{-x_n/4}^{x_n/4} w_{1n}'w_{1n}^{2p_n-3}w_{2n}\hat\psi_ndx+o(e^{-(2+\varepsilon)\sqrt{|\mu_{2n}|}x_n})\nonumber\\
=&-e^{-\sqrt{|\mu_{2n}|}x_n}\inte x\cdotp w_{1n}'w_{1n}^{2p_n-3}w_{2n}w_*dx+o\big(e^{-2\sqrt{|\mu_{2n}|}x_n}\big)\label{4.16}\\
=&\frac{1}{8}e^{-2\sqrt{|\mu_{2n}|}x_n}+o\big(e^{-2\sqrt{|\mu_{2n}|}x_n}\big)
\ \ \ \text{as}\ \ n\to\infty,\nonumber\\
F_n=&-[1+o(1)](\mu_{1n}-\mu_{2n})\inte w_{1n}'\big[b_n^2\hat\phi_n+b_n\sqrt{1-a_n^2}\hat\psi_n\big]dx\label{4.17}\\
=&-\frac{1}{16}e^{-2\sqrt{|\mu_{2n}|}x_n}+o\big(e^{-2\sqrt{|\mu_{2n}|}x_n}\big)
\ \ \ \text{as}\ \ n\to\infty,\nonumber
\end{align}
and
\begin{align}\label{4.77}
G_n=&[1+o(1)]\int_{-3x_n/4}^{x_n/2} w_{1n}'w_{1n}^{2p_n-3}(\hat\phi_n^2+\hat\psi_n^2)dx\nonumber\\
=& [1+o(1)]e^{-2\sqrt{|\mu_{2n}|}x_n}\int_{-3x_n/4}^{x_n/2} w_{1n}'w_{1n}^{2p_n-3}\big[x^2/4+4(xw_*'+w_*)^2\big]dx\\
=&o\big(e^{-2\sqrt{|\mu_{2n}|}x_n}\big)\ \text{as}\ n\to\infty.\nonumber
\end{align}
Applying  Lemma \ref{lemA.1}, we thus deduce from \eqref{4.74}--\eqref{4.77} that
\begin{equation}\label{4.82a}
\begin{split}
&-\frac{1+o(1)}{8}e^{-2\sqrt{|\mu_{2n}|}x_n}=A_n+B_n+C_n\\
=&\frac{1+o(1)}{4}e^{-2\sqrt{|\mu_{2n}|}x_n}+[1+o(1)]\frac{p_n-2}{128}x_n^2e^{-2\sqrt{|\mu_{2n}|}x_n} \ \ \ \text{as}\ \ n\to\infty,\nonumber
\end{split}\end{equation}
which further yields that $ \lim\limits_{n\to\infty}(2-p_n)x_n^2=48. $
This therefore completes the proof of Lemma \ref{lem4.7}.\qed


\subsection{Proofs of Theorem \ref{thm1.1} and Proposition \ref{prop1.2}}

This subsection is devoted to the proofs of Theorem \ref{thm1.1} and Proposition \ref{prop1.2}. We first address the proof of Theorem \ref{thm1.1} as follows.\vspace{0.15cm}

\noindent \textbf{Proof of Theorem \ref{thm1.1}.} Recall from \cite{ii} that $J_2(2)$ does not admit any minimizer. By contradiction, we now suppose that $J_2(p)$ admits minimizers as $p\searrow 2$. Let $(\bar u_{1n}, \bar u_{2n})$ be a minimizer of $J_2(p_n)$, where $p_n\in (2, 5/2]$ satisfies $p_n\searrow 2$ as $n\to\infty$.
We then consider the minimizer sequence $\big\{(u_{1n}(x),  u_{2n}(x))\big\}$ defined by (\ref{2.39}), where $x_n>0$ is a global maximum point of $u_{2n}$ as $n\to\infty$. We thus conclude from Lemma \ref{lem4.7} that $\lim\limits_{n\to\infty}(2-p_n)x_n^2=48$, which however contradicts with the fact $2-p_n<0$. This hence proves the nonexistence of minimizers for $J_2(p)$, provided that the parameter $p-2\ge 0$ is sufficiently small, and we are done.\qed

\vspace{0.15cm}
In order to prove Proposition \ref{prop1.2}, we need the following refined estimates  as $p_n\nearrow2$.

\begin{lem}\label{lem5.1}
Let $a_n$ and $b_n$ be given by \eqref{3.5}. Then we have
	\begin{align}\label{4.20a}
a_n=\frac{\sqrt{2}}{2}+O(x_ne^{-\sqrt{|\mu_{2n}|}x_n}),\ \ \, b_n=\frac{\sqrt{2}}{2}+O(x_ne^{-\sqrt{|\mu_{2n}|}x_n})\ \ \ \text{as}\ \ p_n\nearrow2.
	\end{align}
\end{lem}

\noindent \textbf{Proof.}
We first claim that
\begin{equation}\label{4.82}
\begin{split}
	&\inte xw_{1n}'(\mathcal{L}_n\hat\phi_n)dx \\
	 =&\frac{1}{2}b_n^2(\mu_{1n}-\mu_{2n})\|w_{1n}\|_2^2-\frac{1}{64}x_n^2e^{-2\sqrt{|\mu_{2n}|}x_n}+
O\big(x_ne^{-2\sqrt{|\mu_{2n}|}x_n}\big) \ \,\ \text{as}\ \ p_n\nearrow2,
\end{split}
\end{equation}
where we define the operator $\mathcal{L}_{1n}:=-D_{xx}-\mu_{1n}-(2p_n-1)w_{1n}^{2(p_n-1)}$ in $\R$.
Essentially, since $\lim\limits_{p_n\nearrow2}(2-p_n)x_n=0$,
similar to \eqref{4.70b}, \eqref{4.70c} and \eqref{4.75}, it yields that
\begin{align}\label{4.19}
	&\inte \Big\{xw_{1n}w_{1n}'\big[(u_{1n}^2+u_{2n}^2)^{p_n-1}-w_{1n}^{2p_n-2}\big]-2(p_n-1) xw_{1n}^{2p_n-2}w_{1n}'\hat\phi_n\Big\}dx\nonumber\\
	=&\int_{x_n/2}^{3x_n/2}xw_{1n}w_{1n}'w_{2n}^{2p_n-2}dx
	+(p_n-1)\int_{-3x_n/4}^{x_n/2}xw^{2p_n-3}_{1n}w_{1n}'w_{2n}^2\\
	&+(p_n-1)2K_n\int_{-3x_n/4}^{x_n/2}xw^{2p_n-2}_{1n}w_{1n}'w_{2n}dx+o(x_ne^{-2\sqrt{|\mu_{2n}|}x_n})\nonumber\\
	:=&H_n+2I_n+o\big(x_ne^{-2\sqrt{|\mu_{2n}|}x_n}\big)\ \ \ \text{as}\ \ n\to\infty,\nonumber\\[2mm]
	K_n&\inte xw_{1n}'w_{2n}\big[(u_{1n}^2+u_{2n}^2)^{p_n-1}-w_{2n}^{2p_n-2}\big]dx
	=O\big(x_ne^{-2\sqrt{|\mu_{2n}|}x_n}\big)\ \ \text{as}\ \ p_n\nearrow2,\label{4.20}
\end{align}
 and
\begin{align}\label{4.87}
	&\inte \big[(u_{1n}^2+u_{2n}^2)^{p_n-1}-w_{1n}^{2p_n-2}\big]xw_{1n}'\hat \phi_ndx\\
	=&\int_{B_{x_n/4}(x_n)} w_{2n}^{2(p_n-1)}xw_{1n}'\hat \phi_ndx+o(x_ne^{-2\sqrt{|\mu_{2n}|}x_n})
	=O\big(x_ne^{-2\sqrt{|\mu_{2n}|}x_n}\big)\ \ \ \text{as}\ \ p_n\nearrow2.\nonumber
\end{align}
Following Lemma \ref{lemA.2} and the fact that $\lim\limits_{p_n\nearrow2}(2-p_n)x_n^2=48$,
we then deduce  from \eqref{AA3.5} and \eqref{4.19}--\eqref{4.87} that
\begin{align}\label{4.53a}
	&\inte(\mathcal{L}_n\hat\phi_n)xw_{1n}'dx\nonumber\\
	=&\inte \big[(u_{1n}^2+u_{2n}^2)^{p_n-1}-w_{1n}^{2p_n-2}\big]xw_{1n}'\hat \phi_ndx\nonumber\\
	&+\inte xw_{1n}w_{1n}'\big[(u_{1n}^2+u_{2n}^2)^{p_n-1}-w_{1n}^{2p_n-2}\big]dx-(2p_n-2)\inte  w_{1n}^{2p_n-2}xw_{1n}'\hat\phi_ndx\nonumber\\
	&+K_n\inte xw_{1n}'w_{2n}\big[(u_{1n}^2+u_{2n}^2)^{p_n-1}-w_{2n}^{2p_n-2}\big]dx-b_n(\mu_{1n}-\mu_{2n})\inte xw_{1n}'\psi_ndx\nonumber\\
	&+a_n\sqrt{1-b_n^2}(\mu_{1n}-\mu_{2n})\inte xw_{1n}'w_{2n}dx -b_n^2(\mu_{1n}-\mu_{2n})\inte xw_{1n}w_{1n}'dx\\
	=&H_n+3I_n+a_n\sqrt{1-b_n^2}(\mu_{1n}-\mu_{2n})\inte xw_{1n}'w_{2n}dx \nonumber\\
	&-b_n^2(\mu_{1n}-\mu_{2n})\inte xw_{1n}w_{1n}'dx+O\big(x_ne^{-2\sqrt{|\mu_{2n}|}x_n}\big)\nonumber\\
	=&-\frac{1}{64}x_n^2e^{-2\sqrt{|\mu_{2n}|}x_n}-b_n^2(\mu_{1n}-\mu_{2n})\inte xw_{1n}w_{1n}'dx+O\big(x_ne^{-2\sqrt{|\mu_{2n}|}x_n}\big)\nonumber
\end{align}
as $p_n\nearrow2$, which thus yields the claim \eqref{4.82}.

Similarly, one can verify that
\begin{align}\label{4.54}
	&\frac{1}{p_n-1}\inte w_{1n}(\mathcal{L}_{1n}\hat\phi_n)dx\nonumber\\
	=&\int_{x_n/2}^{3x_n/2} w_{1n}^2w_{2n}^{2p_n-2}dx+\int^{x_n/2}_{-3x_n/4} w_{1n}^{2p_n-2}w_{2n}^2dx+3K_n\int^{x_n/2}_{-3x_n/4} w_{1n}^{2p_n-1}w_{2n}dx\\
	&+\frac{a_n\sqrt{1-b_n^2}(\mu_{1n}-\mu_{2n})}{p_n-1}\inte w_{1n}w_{2n}dx -\frac{b_n^2(\mu_{1n}-\mu_{2n})}{p_n-1}\inte w_{1n}^2dx+o\big(x_ne^{-2\sqrt{|\mu_{2n}|}x_n}\big)\nonumber\\
	=&-\frac{b_n^2(\mu_{1n}-\mu_{2n})}{p_n-1}\|w_{1n}\|_2^2+O\big(x_ne^{-2\sqrt{|\mu_{2n}|}x_n}\big)\ \ \ \text{as}\ \ p_n\nearrow2.\nonumber
\end{align}
Applying \eqref{4.57} and \eqref{a.3}, we  then obtain from \eqref{4.82} and \eqref{4.54} that
\begin{align}\label{4.55}
	\inte w_{1n}\hat \phi_ndx
	=&\frac{1}{2\mu_{1n}}\inte \mathcal{L}_{1n}\hat\phi_n\big(xw_{1n}'+\frac{1}{p_n-1}w_{1n}\big)dx\nonumber\\
	 =&b_n^2\big(\mu_{1n}-\mu_{2n}\big)\frac{3-p_n}{4(p_n-1)|\mu_{1n}|}\|w_{1n}\|_2^2+\frac{1}{8}x_n^2e^{-2\sqrt{|\mu_{2n}|}x_n}\\
	&+O\big(x_ne^{-2\sqrt{|\mu_{2n}|}x_n}\big)\ \ \ \text{as}\ \ p_n\nearrow2.\nonumber
\end{align}
Similar to \eqref{4.55}, it gives from \eqref{4.58} that
\begin{equation}\label{4.56}
\begin{split}
	\inte w_{2n}\hat \psi_ndx
=&-a_n^2\big(\mu_{1n}-\mu_{2n}\big)\frac{3-p_n}{4(p_n-1)|\mu_{2n}|}\|w_{2n}\|_2^2+\frac{1}{8}x_n^2e^{-2\sqrt{|\mu_{2n}|}x_n}\\
	&+O\big(x_ne^{-2\sqrt{|\mu_{2n}|}x_n}\big)\ \ \ \text{as}\ \ p_n\nearrow2.
\end{split}
\end{equation}

For simplicity we now denote $\alpha_n:=\frac{3-p_n}{2(p_n-1)}>0$. Since   $\|w_{2n}\|_2^2=|\mu_{2n}|^{\alpha_n}\|\widetilde{w}\|_2^2$ and
\begin{align*}
	\|w_{1n}\|_2^2=&|\mu_{1n}|^{\alpha_n}\|\widetilde{w}\|_2^2\\ =&\big[|\mu_{2n}|^{\alpha_n}+\alpha_n|\mu_{2n}|^{\alpha_n-1}(\mu_{2n}-\mu_{1n})+\frac{\alpha_n(\alpha_n-1)}{2}|\mu_{2n}|^{\alpha_n-2}(\mu_{2n}-\mu_{1n})^2\big]\|\widetilde{w}\|_2^2\\
	&+o\big((\mu_{2n}-\mu_{1n})^2\big)\ \ \ \text{as}\ \ p_n\nearrow2,
\end{align*}
we deduce from \eqref{4.55} and \eqref{4.56} that
\begin{align}\label{4.89}
	&2\inte \big(w_{1n}\hat \phi_n-w_{2n}\hat \psi_n\big)dx\nonumber\\
	=&\alpha_n|\mu_{2n}|^{\alpha_n-1}\|\widetilde{w}\|_2^2(a_n^2+b_n^2)(\mu_{1n}-\mu_{2n})
	+O\big(x_ne^{-2\sqrt{|\mu_{2n}|}x_n}\big) \ \ \text{as}\ \ p_n\nearrow2,\nonumber
\end{align}
due to the fact that $|\mu_{1n}-\mu_{2n}|=O(e^{-\sqrt{|\mu_{2n}|}x_n})$ as $p_n\nearrow2$.
We then calculate  from \eqref{4.42} that
\begin{align*}
	&-\alpha_n|\mu_{2n}|^{\alpha_n-1}\|\widetilde{w}\|_2^2(a_n^2+b_n^2)(\mu_{1n}-\mu_{2n})\nonumber\\
	=&-2\inte \big(w_{1n}\hat \phi_n-w_{2n}\hat \psi_n\big)dx+O\big(x_ne^{-2\sqrt{|\mu_{2n}|}x_n}\big)\\
	=&\inte(w_{1n}^2-w_{2n}^2)dx+O\big(x_ne^{-2\sqrt{|\mu_{2n}|}x_n}\big) \nonumber\\ =&\Big[\alpha_n|\mu_{2n}|^{\alpha_n-1}(\mu_{2n}-\mu_{1n})+\frac{\alpha_n(\alpha_n-1)}{2}|\mu_{2n}|^{\alpha_n-2}(\mu_{2n}-\mu_{1n})^2\Big]\|\widetilde{w}\|_2^2+O\big(x_ne^{-2\sqrt{|\mu_{2n}|}x_n}\big) \nonumber
\end{align*}
as $p_n\nearrow2$, which further yields that
\begin{equation}\label{4.30a}
	\begin{split}a_n^2+b_n^2
		=&1+\frac{(\alpha_n-1)}{2}|\mu_{2n}|^{-1}(\mu_{2n}-\mu_{1n})+O\big(x_ne^{-\sqrt{|\mu_{2n}|}x_n}\big)\\
		=&1+ O\big(x_ne^{-\sqrt{|\mu_{2n}|}x_n}\big)\ \ \ \text{as}\ \ p_n\nearrow2.
\end{split}\end{equation}
Set $a_n:=\frac{\sqrt{2}}{2}+\tilde{a}_n$ and $b_n:=-\frac{\sqrt{2}}{2}+\tilde{b}_n$.
Since $a_n+b_n=\frac{1}{\sqrt{2}}K_n=-\frac{1}{2\sqrt{2}}x_ne^{-\sqrt{|\mu_{2n}|}x_n}$ as $p_n\nearrow2$,  we deduce from \eqref{4.30a} that $\tilde{a}_n-\tilde{b}_n=O(K_n)$ as $p_n\nearrow2$,  which thus yields that \eqref{4.20a} holds true. The proof of Lemma \ref{lem5.1} is therefore complete.\qed

Following Lemma \ref{lem5.1}, we are ready to complete the proof of Proposition \ref{prop1.2}.

\vspace {.15cm}
\noindent \textbf{Proof of Proposition \ref{prop1.2}.}  Let  $(u_{1n}, u_{2n})$ be a minimizer of $J_2(p_n)$, where  $p_n\nearrow 2$ as $n\to\infty$,  $u_{1n}(0)=\max_{\R}u_{1n}(x)$ and $u_{2n}(0)\leq0$.  Recall from Sections 2 and  3  that  up to a subsequence if necessary, $(u_{1n}, u_{2n})$ satisfies \eqref{4.10a} in the sense that
\begin{eqnarray}\label{4}
\left\{
\begin{array}{lll}
\!	\!	\!	 u_{1n}(x)=\sqrt{1-b^2_n}\widetilde w_{1n}(x-\delta_n)+a_n\widetilde w_{2n}(x- x_n+\eta_n)+ \phi_n(x), \\[1mm]
\!	\!	\!	 u_{2n}(x)=b_n\widetilde w_{1n}(x-\delta_n)+\sqrt{1-a_n^2}\widetilde w_{2n}(x- x_n+\eta_n)+ \psi_n(x),
\end{array}
\right.
\end{eqnarray}
where 
 $\lim\limits_{n\to\infty}(a_n, b_n, \delta_n, \eta_n)= (\frac{\sqrt{2}}{2}, -\frac{\sqrt{2}}{2}, 0, 0)$ is as in \eqref{3.5},  and $ x_n\in\R^+$ satisfying $x_n\to\infty$ as $n\to\infty$ denotes the maximum point  of $ u_{2n}$  in $\R$. Applying Lemma \ref{a}, we obtain from \eqref{AA3.2} and \eqref{0.10} that
\begin{eqnarray*}
|\delta_n|+|\eta_n|+\|\phi_n\|_\infty+\|\psi_n\|_\infty=O(e^{-\sqrt{|\mu_{2n}|}\,x_n})\ \ \text{as}\ \ n\to\infty,
\end{eqnarray*}
where $\mu_{2n}$ is as in \eqref{1:win}.
We then conclude from \eqref{1:win} and \eqref{4} that
\begin{eqnarray}\label{4.34a}
	\left\{
	\begin{array}{lll}
		\!	\!	\!	 u_{1n}(x)=\sqrt{1-b^2_n}\widetilde w_{1n}(x)+a_n\widetilde w_{2n}(x- x_n)+ O(e^{-\sqrt{|\mu_{2n}|}\,x_n})\ \ \ \text{as}\ \ n\to\infty, \\[1mm]
		\!	\!	\!	 u_{2n}(x)=b_n\widetilde w_{1n}(x)+\sqrt{1-a_n^2}\widetilde w_{2n}(x- x_n)+ O(e^{-\sqrt{|\mu_{2n}|}\,x_n})\ \ \ \text{as}\ \ n\to\infty.
	\end{array}
	\right.
\end{eqnarray}
Applying Lemma \ref{lem5.1},  we further deduce from \eqref{4.34a} and \eqref{a.3} that
\begin{eqnarray*}
	\left\{
	\begin{array}{lll}
		\!	\!	\!	 u_{1n}(x)=\frac{\sqrt{2}}{2}\widetilde w_{1n}(x)+\frac{\sqrt{2}}{2}\widetilde w_{2n}(x- x_n)+ O(x_ne^{-\frac{x_n}{4}})\ \ \ \text{as}\ \ n\to\infty, \\[1mm]
		\!	\!	\!	 u_{2n}(x)=-\frac{\sqrt{2}}{2}\widetilde w_{1n}(x)+\frac{\sqrt{2}}{2}\widetilde w_{2n}(x- x_n)+ O(x_ne^{-\frac{x_n}{4}})\ \ \ \text{as}\ \ n\to\infty,
	\end{array}
	\right.
\end{eqnarray*}
which  therefore completes the proof of Proposition \ref{prop1.2}.
\qed

\appendix
\section{Appendix}

In this appendix, we shall establish Lemmas \ref{lemA.1} and \ref{lemA.2}, which are used in Lemmas  \ref{lem4.7} and \ref{lem5.1}, respectively. 

\begin{lem}\label{lemA.1}
Let $w_{1n}$ and  $w_{2n}$ be defined by  \eqref{3.16a}, and suppose   $\lim\limits_{n\to\infty}(2-p_n)x_n=0$, where $p_n\to2$ as $n\to\infty$, and $x_n>0$ is as in \eqref{3.16a}. Then we have
\begin{align}\label{b.1}
A_n:=&\int_{x_n/2}^{3x_n/2}w_{1n}w_{1n}'w_{2n}^{2p_n-2}dx
+(p_n-1)\int_{-3x_n/4}^{x_n/2}w^{2p_n-3}_{1n}w_{1n}'w_{2n}^2dx\\
=&-\frac{1}{16}\Big[x_n-6+\big(1+o(1)\big)\frac{2-p_n}{2}\sqrt{|\mu_{2n}|}x_n^2+o(1)\Big]e^{-2\sqrt{|\mu_{2n}|}x_n}\ \ \ \text{as}\ \ n\to\infty, \nonumber
\end{align}
\begin{equation}\label{b.3}
\begin{split}
B_n:=&(p_n-1)2K_n\int_{-3x_n/4}^{x_n/2}w_{1n}'w_{2n}w^{2p_n-2}_{1n}dx+K_n\int_{-x_n/2}^{x_n/2}w_{1n}'
w_{2n}w_{1n}^{2p_n-2}dx\\
=&\Big[\frac{1}{32}x_n+o(1)\Big]e^{-2\sqrt{|\mu_{2n}|}x_n}\ \ \ \text{as}\ \ n\to\infty,
\end{split}
\end{equation}
\begin{equation}\label{b.2}
\begin{split}
C_n:=&-a_n\sqrt{1-b_n^2}(\mu_{2n}-\mu_{1n})\inte w_{2n}w_{1n}'dx\\
=&\Big[\frac{1}{32}x_n-\frac{1}{8}+o(1)\Big]e^{-2\sqrt{|\mu_{2n}|}x_n}\ \ \ \text{as}\ \ n\to\infty. \qquad\qquad\qquad\qquad\qquad
\end{split}
\end{equation}
\end{lem}

\noindent \textbf{Proof.}
We first derive from  \eqref{4.10} and Lemma \ref{a} that
\begin{eqnarray*}
\left\{
\begin{array}{lll}
\!	\!	\!	 \|w_{1n}\|_2^2=1-2\inte w_{1n}\hat\phi_ndx+o\big(e^{-\sqrt{|\mu_{2n}|}x_n}\big)\ \  \ \text{as} \  \ n\to\infty,\\[1mm]
\!	\!	\!	\|w_{2n}\|_2^2=1-2\inte w_{2n}\hat\psi_ndx+o\big(e^{-\sqrt{|\mu_{2n}|}x_n}\big)\ \  \ \text{as} \  \ n\to\infty.
\end{array}
\right.
\end{eqnarray*}
It then follows  from \eqref{4.59} and \eqref{4.60} that
\begin{align}\label{b.4}
\|w_{1n}\|_2^2=1+[2+o(1)]e^{-\sqrt{|\mu_{2n}|}x_n},\ \ \ \|w_{2n}\|_2^2=1-[2+o(1)]e^{-\sqrt{|\mu_{2n}|}x_n}
\end{align}
as $n\to\infty$. It yields from \eqref{w} that
$$
\|\widetilde{w}_n\|_2^2=4+(6-12\text{ln}2)(p_n-2)+o(p_n-2)\ \ \ \text{as}\ \ n\to\infty,
$$
which further gives from \eqref{win} and \eqref{b.4}  that for i=1,2,
\begin{align}\label{a.3} \sqrt{|\mu_{in}|}=&\big(\|w_{in}\|_2^2\|\widetilde{w}_n\|_2^{-2}\big)^{\frac{p_n-1}{3-p_n}}
=\frac{1}{4}-\Big(\frac{6-\text{ln}12}{16}+\text{ln}2\Big)(p_n-2)+o(p_n-2)
\end{align}
as $n\to\infty$.

Let  $\delta_n$ and $\eta_n$  be as in \eqref{3.5}. We claim that
\begin{eqnarray}\label{0.10}
	\delta_n= O(e^{-\sqrt{|\mu_{2n}|}\, x_n}),\ \ \, \eta_n=O(e^{-\sqrt{|\mu_{2n}|}\,x_n})\ \ \,\text{as}\ \ n\to\infty.
\end{eqnarray}
Actually, it follows from \eqref{3.2} and \eqref{4.10a} that
\begin{eqnarray}\label{a.9}
	\left\{
	\begin{array}{lll}
		\!	\!	\!	 0=u_{1n}'(0)=\sqrt{1-b^2_n}\widetilde w_{1n}'(-\delta_n)+a_n\widetilde w_{2n}'(-x_n+\eta_n)+ \phi_n'(0)\\
		\ \ \ \ \ \ \ \ \ \ \ \ \ =-\delta_n\sqrt{1-b^2_n}\widetilde w_{1n}''(0)+a_n\widetilde w_{2n}'(-x_n+\eta_n)+ \phi_n'(0), \\[1.5mm]
		\!	\!	\!	 0=u_{2n}'(x_n)=b_n\widetilde w_{1n}'(x_n-\delta_n)+\sqrt{1-a_n^2}\widetilde w_{2n}'(\eta_n)+ \psi_n'(x_n)\\
		\ \ \ \ \ \ \ \ \ \ \ \ \ \  \ =b_n\widetilde w_{1n}'(x_n-\delta_n)+\eta_n\sqrt{1-a_n^2}\widetilde w_{2n}'(0)+ \psi_n'(x_n).
	\end{array}
	\right.
\end{eqnarray}
Since
\begin{eqnarray*}
\left\{
\begin{array}{lll}
\!	\!	\!	\widetilde w'_{in}(x)=\Big(2\sqrt{|\mu_{in}|p_n}\Big)^{\frac{1}{p_n-1}}\sqrt{|\mu_{in}|}
\frac{e^{\sqrt{|\mu_{in}|}x}\big(1-e^{2(p_n-1)\sqrt{|\mu_{in}|}x}\big)}{\big(1+e^{2(p_n-1)\sqrt{|\mu_{in}|}x}\big)^{\frac{p_n}{p_n-1}}}\ \ \ \text{in}\,\ \R, \\[1mm]
\!	\!	\!	 w^{''}_{in}(x)=\Big(2\sqrt{|\mu_{in}|p_n}\Big)^{\frac{1}{p_n-1}}|\mu_{in}|e^{\sqrt{|\mu_{in}|}x}\big(1+e^{2(p_n-1)\sqrt{|\mu_{in}|}x}\big)^{\frac{-2p_n+1}{p_n-1}}\\[1.5mm]
\ \ \ \ \ \ \ \  \ \ \ \cdot\left[\Big(1-(2p_n-1)e^{2(p_n-1)\sqrt{|\mu_{in}|}x}\Big)\big(1+e^{2(p_n-1)\sqrt{|\mu_{in}|}x}\big)\right.\\[1mm]
\ \ \ \ \ \ \ \ \ \ \ \ \ \ \left.-2p_n\big(1-e^{2(p_n-1)\sqrt{|\mu_{in}|}x}\big)e^{2(p_n-1)\sqrt{|\mu_{in}|}x}\right]\ \ \ \text{in}\,\ \R,
\end{array}
\right.
\end{eqnarray*}
we have
\begin{eqnarray}\label{a.6}
\left\{
\begin{array}{lll}
\!	\!	\!	\widetilde w'_{in}(x)=-\displaystyle\frac{1+o(1)}{4\sqrt{2}}e^{-\sqrt{|\mu_{in}|}\,|x|}\ &\text{as}\ \ |x|\to\infty, \\[2mm]
\!	\!	\!	\widetilde w''_{in}(x)=-\displaystyle\frac{1+o(1)}{32\sqrt{2}}\ &\text{as}\ \ |x|\to0.
\end{array}
\right.
\end{eqnarray}
Using  Lemma \ref{a} and the local elliptic estimates (cf. \cite[Equation (3.15)]{elli}),   we deduce from \eqref{AA3.5} and \eqref{AA3.6}  that
\begin{equation*}\label{b.11}
|\phi_n'(0)|+ |\psi_n'(x_n)|=O(e^{-\sqrt{|\mu_{2n}|}\,x_n})\ \ \ \text{as}\ \ n\to\infty,
\end{equation*}
which thus yields from \eqref{a.9} and  \eqref{a.6}  that the claim \eqref{0.10} holds true.

We now estimate $C_n$. It yields from \eqref{4.44b}, \eqref{O4} and \eqref{O5} that
\begin{align}\label{b.6}
&a_n\sqrt{1-b_n^2}(\mu_{2n}-\mu_{1n})\|w_{2n}\|_2^2\nonumber\\
=&\inte \big[(u_{1n}^2+u_{2n}^2)^{p_n-1}-w_{1n}^{2p_n-2}\big]w_{1n}w_{2n}dx+o\big(e^{-\frac{5}{4}\sqrt{|\mu_{2n}|}x_n}\big)\\
=&\int_{x_n/2}^{3x_n/2}w_{1n}w_{2n}^{2p_n-1}dx+o\big(e^{-\frac{5}{4}\sqrt{|\mu_{2n}|}x_n}\big)\ \ \ \text{as}\ \ n\to\infty.\nonumber
\end{align}
Since  $\lim\limits_{n\to\infty}(2-p_n)x_n=0$,   using   Taylor's expansion,   there exists $\theta_n\in(0,1)$ such that
\begin{align*}
&\frac{1}{[1+e^{-2(p_n-1)\sqrt{|\mu_{2n}|}x}]^{\frac{2p_n-1}{p_n-1}}}\\
=&\frac{1+\frac{p_n-2}{p_n-1}(1+e^{-2(p_n-1)\sqrt{|\mu_{2n}|}x})^{\theta_n\frac{p_n-2}{p_n-1}}\text{ln}[1+e^{-2(p_n-1)
\sqrt{|\mu_{2n}|}x}]}{[1+e^{-2(p_n-1)\sqrt{|\mu_{2n}|}x}]^{3}}\\
=&\frac{1+(1+o(1))(p_n-2)\text{ln}[1+e^{-2(p_n-1)\sqrt{|\mu_{2n}|}x}]}{[1+e^{-2(p_n-1)\sqrt{|\mu_{2n}|}x}]^{3}}
\end{align*}
uniformly in $[-x_n/2,\, x_n/2]$ as $n\to\infty$. 
We thus calculate from \eqref{4.63},  \eqref{a.3} and \eqref{0.10}  that 
\begin{align}\label{b.8}
&\int_{x_n/2}^{3x_n/2}w_{1n}w_{2n}^{2p_n-1}dx\nonumber\\
=&\big[\frac{1}{4}+O(p_n-2)\big] e^{-\sqrt{|\mu_{2n}|}\,x_n}\int_{-x_n/2}^{x_n/2} \frac{e^{-2(p_n-1)\sqrt{|\mu_{2n}|}\,x}}{[1+e^{-2(p_n-1)\sqrt{|\mu_{2n}|}\,x}]^{\frac{2p_n-1}{p_n-1}}}dx\nonumber\\
=&
\big[\frac{1}{4}+O(p_n-2)\big]e^{-\sqrt{|\mu_{2n}|}\,x_n}\int_{-x_n/2}^{x_n/2}
\frac{e^{-2(p_n-1)\sqrt{|\mu_{2n}|}x}\big[1+(p_n-2)\text{ln}(1+e^{-2(p_n-1)\sqrt{|\mu_{2n}|}\,x})\big]}
{(1+e^{-2(p_n-1)\sqrt{|\mu_{2n}|}\,x})^{3}}dx\nonumber\\
=&
\big[\frac{1}{4}+O(p_n-2)\big]e^{-\sqrt{|\mu_{2n}|}\,x_n}\int_{-x_n/2}^{x_n/2}\frac{e^{-2(p_n-1)\sqrt{|\mu_{2n}|}x}}
{[1+e^{-2(p_n-1)\sqrt{|\mu_{2n}|}\,x}]^{3}}dx\\
=&
\big[\frac{1}{4}+O(p_n-2)\big]e^{-\sqrt{|\mu_{2n}|}\,x_n}\frac{-1}{4(p_n-1)\sqrt{|\mu_{2n}|}}
\frac{2e^{2(p_n-1)\sqrt{|\mu_{2n}|}\,x}+1}{[1+e^{2(p_n-1)\sqrt{|\mu_{2n}|}\,x}]^2}\Big|_{-x_n/2}^{x_n/2}\nonumber\\
=&\Big[\frac{1}{4}+O(p_n-2)\Big]e^{-\sqrt{|\mu_{2n}|}\,x_n}\ \ \ \text{as}\ \ n\to\infty,\nonumber
\end{align}
where  we have used the fact that
$$
\lim\limits_{n\to\infty}\int_{-x_n/2}^{x_n/2}\frac{e^{-2(p_n-1)\sqrt{|\mu_{2n}|}x}\text{ln}(1+e^{-2(p_n-1)\sqrt{|\mu_{2n}|}x})}{(1+e^{-2(p_n-1)\sqrt{|\mu_{2n}|}x})^{3}}dx=\frac{1}{2}.
$$
As a consequence of \eqref{b.4},  \eqref{b.6} and \eqref{b.8}, we obtain that
\begin{align}\label{b.8a}
-a_n\sqrt{1-b_n^2}(\mu_{2n}-\mu_{1n})
=&-\Big[\frac{1}{4}+O(p_n-2)\Big]e^{-\sqrt{|\mu_{2n}|}x_n}\ \ \, \text{as}\ \ n\to\infty.
\end{align}

Moreover, similar to \eqref{b.8}, one can calculate that
\begin{equation*}
\begin{split}
&\int_{-\infty}^{0} w'_{1n}w_{2n}dx\\
=&\big[\frac{1}{8}+O(p_n-2)\big]e^{-\sqrt{|\mu_{2n}|}x_n}\int_{-\infty}^{0} \frac{e^{(\sqrt{|\mu_{1n}|}+\sqrt{|\mu_{2n}|})\,x}(1-e^{2\sqrt{|\mu_{1n}|}(p_n-1)x})}
{(1+e^{2\sqrt{|\mu_{1n}|}(p_n-1)x})^2}dx\\
=&\big[\frac{1}{8}+O(p_n-2)\big]e^{-\sqrt{|\mu_{2n}|}x_n}\\
&\cdot\int_{-x_n}^{0}
\frac{[1+(4-2p_n)\sqrt{|\mu_{1n}|}x]e^{2(p_n-1)\sqrt{|\mu_{1n}|}\,x}(1-e^{2\sqrt{|\mu_{1n}|}(p_n-1)x})}
{(1+e^{2\sqrt{|\mu_{1n}|}(p_n-1)x})^2}dx\\
=&\big[\frac{1}{8}+O(p_n-2)\big]e^{-\sqrt{|\mu_{2n}|}x_n}\frac{1}{2\sqrt{|\mu_{1n}|}(p_n-1)}\\
&\cdot\Big[-\frac{2}{1+e^{2\sqrt{|\mu_{1n}|}(p_n-1)x}}-\text{ln}(1+e^{2\sqrt{|\mu_{1n}|}(p_n-1)x})\Big]\Big|_{-x_n}^{0}\\
=&\big[1+O(p_n-2)\big]\frac{1-\text{ln}2}{4}e^{-\sqrt{|\mu_{2n}|}x_n}\ \ \ \text{as}\ \ n\to\infty,
\end{split}
\end{equation*}
\begin{equation*}
\begin{split}
&\int_{0}^{x_n/2} w'_{1n}w_{2n}dx \\
=&\big[\frac{1}{8}+O(p_n-2)\big]e^{-\sqrt{|\mu_{2n}|}x_n} \int_{0}^{x_n/2}\frac{(e^{-2(p_n-1)\sqrt{|\mu_{2n}|}x}-1)}{(1+e^{-2(p_n-1)\sqrt{|\mu_{2n}|}x})^{\frac{p_n}{p_n-1}}}dx\\
=&\big[-\frac{1}{8}+O(p_n-2)\big]e^{-\sqrt{|\mu_{2n}|}x_n}\int_{0}^{x_n/2}
\frac{(1-e^{-2(p_n-1)\sqrt{|\mu_{2n}|}x})}{(1+e^{-2(p_n-1)\sqrt{|\mu_{2n}|}x})^{2}}dx \\
=&\big[-\frac{1}{8}+O(p_n-2)\big]e^{-\sqrt{|\mu_{2n}|}x_n}\frac{1}{2(p_n-1)\sqrt{|\mu_{2n}|}} \\
&\cdot\Big[\frac{2}{1+e^{2(p_n-1)\sqrt{|\mu_{2n}|}x}}+\text{ln}(1+e^{2(p_n-1)\sqrt{|\mu_{2n}|}x})\Big]\Big|_{0}^{x_n/2} \\
=&\Big[-\frac{x_n}{16}+\frac{1}{4}+\frac{\text{ln}2}{4}+O(p_n-2)\Big]e^{-\sqrt{|\mu_{2n}|}x_n}\ \ \ \text{as}\ \ n\to\infty,
\end{split}
\end{equation*}
\begin{equation*}
\begin{split}
&\int^{x_n}_{x_n/2} w'_{1n}w_{2n}dx\nonumber\\
=&\big[-\frac{1}{8}+O(p_n-2)\big]e^{-\sqrt{|\mu_{2n}|}x_n}\int^{x_n}_{x_n/2}\frac{1}{[1+e^{2(p_n-1)
\sqrt{|\mu_{2n}|}(x-x_n)}]^{\frac{1}{p_n-1}}}dx \\
=&\big[-\frac{1}{8}+O(p_n-2)\big]e^{-\sqrt{|\mu_{2n}|}x_n}\int_{0}^{x_n/2}\frac{1+O(p_n-2)\text{ln}(1+e^{-2(p_n-1)\sqrt{|\mu_{2n}|}y})}{1+e^{-2(p_n-1)\sqrt{|\mu_{2n}|}y}}dy\nonumber\\
=&\Big[-\frac{x_n}{16}+\frac{\text{ln}2}{4}+O(p_n-2)\Big]e^{-\sqrt{|\mu_{2n}|}x_n}\ \ \ \text{as}\ \ n\to\infty,\nonumber
\end{split}
\end{equation*}
\begin{equation*}
\begin{split}
&\int^{\infty}_{x_n} w'_{1n}w_{2n}dx=\int^{2x_n}_{x_n} w'_{1n}w_{2n}dx \\
=&\big[-\frac{1}{8}+O(p_n-2)\big]e^{-\sqrt{|\mu_{2n}|}x_n}\int^{2x_n}_{x_n}\frac{1}{[1+e^{2(p_n-1)
\sqrt{|\mu_{2n}|}(x-x_n)}]^{\frac{1}{p_n-1}}}dx \\
=&\big[-\frac{1}{8}+O(p_n-2)\big]e^{-\sqrt{|\mu_{2n}|}x_n}\int^{0}_{-x_n}\frac{1+O(|p_n-2|)
\text{ln}(1+e^{-2(p_n-1)\sqrt{|\mu_{2n}|}y})}{1+e^{-2(p_n-1)\sqrt{|\mu_{2n}|}y}}dy \\
=&\big[-\frac{1}{8}+O(p_n-2)\big]e^{-\sqrt{|\mu_{2n}|}x_n}\frac{1}{2(p_n-1)\sqrt{|\mu_{2n}|}}
\text{ln}(1+e^{2(p_n-1)\sqrt{|\mu_{2n}|}y})\Big|^{0}_{-x_n} \\
=&\Big[-\frac{\text{ln}2}{4}+O(p_n-2)\Big]e^{-\sqrt{|\mu_{2n}|}x_n}\ \   \text{as}\ \ n\to\infty.
\end{split}
\end{equation*}
We then obtain from above that
\begin{align}\label{b.12}
\inte w'_{1n}w_{2n}dx
=\Big[-\frac{x_n}{8}+\frac{1}{2}+O(p_n-2)\Big]e^{-\sqrt{|\mu_{2n}|}x_n}\ \   \text{as}\ \ n\to\infty.
\end{align}
This therefore proves \eqref{b.2} in view of \eqref{b.8a}.

In order to prove \eqref{b.3},  we first deduce from \eqref{4.36a}, \eqref{4.69a} and \eqref{4.73} that
\begin{equation}\label{a.18}
\begin{split}
K_n
=&-\inte w_{1n}w_{2n} dx+o\big(e^{-\sqrt{|\mu_{2n}|}x_n}\big)\ \ \ \text{as}\ \ n\to\infty.
\end{split}\end{equation}
The same arguments of \eqref{b.8} and \eqref{b.12} show that
\begin{equation}\label{a15}
-\inte w_{1n}w_{2n} dx=\Big[-\frac{x_n}{2}+O(p_n-2)\Big]e^{-\sqrt{|\mu_{2n}|}x_n}\ \ \text{as}\ \ n\to\infty,
\end{equation}
and
\begin{align}\label{a.16}
&(2p_n-1)\inte w_{2n}w'_{1n}w_{1n}^{2p_n-2}dx\nonumber\\
=&O(e^{-\frac{3}{2}\sqrt{|\mu_{2n}|}x_n})-\int_{-x_n/4}^{x_n/4}w'_{2n}w_{1n}^{2p_n-1}dx\nonumber\\
=&-\big[\frac{1}{16}+O(p_n-2)\big]e^{-\sqrt{|\mu_{2n}|}x_n}\int_{-x_n/4}^{x_n/4}\frac{e^{-2(p_n-1)\sqrt{|\mu_{2n}|}x}}{\big(1+e^{-2(p_n-1)\sqrt{|\mu_{2n}|}x}\big)^{\frac{2p_n-1}{p_n-1}}}dx\\
=&-\big[\frac{1}{16}+O(p_n-2)\big]e^{-\sqrt{|\mu_{2n}|}x_n}\int_{-x_n/4}^{x_n/4}\frac{e^{-2(p_n-1)\sqrt{|\mu_{2n}|}x}}{\big(1+e^{-2(p_n-1)\sqrt{|\mu_{2n}|}x}\big)^{3}}dx\nonumber\\
=&\Big[\frac{1}{16}+O(p_n-2)\Big]\frac{e^{-\sqrt{|\mu_{2n}|}x_n}}{4(p_n-1)
\sqrt{|\mu_{2n}|}}\frac{2e^{2(p_n-1)\sqrt{|\mu_{2n}|}x}+1}{(1+e^{2(p_n-1)\sqrt{|\mu_{2n}|}x})^2}
\Big|_{-x_n/4}^{x_n/4}\nonumber\\
=&-\Big[\frac{1}{16}+O(p_n-2)\Big]e^{-\sqrt{|\mu_{2n}|}x_n}\ \ \ \text{as}\ \ n\to\infty.\nonumber
\end{align}
This proves \eqref{b.3}.

We next prove \eqref{b.1} as follows.  Since
\begin{equation}\label{a.13a}
\begin{split}
&\int^{x_n/2}_{-x_n/2}\frac{x}{(1+e^{-2(p_n-1)\sqrt{|\mu_{2n}|}x})^{2}}dx=\int^{x_n/2}_{0}
\frac{x}{(1+e^{-2(p_n-1)\sqrt{|\mu_{2n}|}x})^{2}}dx+O(1)\\
=&\int^{x_n/2}_{0}\big[x+O(xe^{-2(p_n-1)\sqrt{|\mu_{2n}|}x})\big]dx=\frac{x_n^2}{8}+O(1)\ \ \ \text{as}\ \ n\to\infty,
\end{split}
\end{equation}
and
\begin{equation}\label{a.15}
\begin{split}
&\int^{x_n/2}_{-3x_n/4}x\frac{1-e^{-2(p_n-1)\sqrt{|\mu_{2n}|}x}}{(1+e^{-2(p_n-1)\sqrt{|\mu_{2n}|}x})^{3}}dx \\
=&\int^{x_n/2}_{0}x\frac{1-e^{-2(p_n-1)\sqrt{|\mu_{2n}|}x}}{(1+e^{-2(p_n-1)\sqrt{|\mu_{2n}|}x})^{3}}dx+O(1)
=\frac{x_n^2}{8}+O(1)\ \ \ \text{as}\ \ n\to\infty,
\end{split}
\end{equation}
we have
\begin{align}\label{a.16a}
&\int_{x_n/2}^{3x_n/2}w_{1n}w_{1n}'w_{2n}^{2p_n-2}dx\nonumber\\
=&-\frac{1+O(p_n-2)}{16}\int^{3x_n/2}_{x_n/2}\frac{e^{-2\sqrt{|\mu_{2n}|}x}e^{2(p_n-1)
\sqrt{|\mu_{2n}|}(x-x_n)}}{(1+e^{2(p_n-1)\sqrt{|\mu_{2n}|}(x-x_n)})^2}dx\nonumber\\
=&-\frac{1+O(p_n-2)}{16}e^{-(2p_n-2)\sqrt{|\mu_{2n}|}x_n}
\int^{x_n/2}_{-x_n/2}\frac{1-(4-2p_n)\sqrt{|\mu_{2n}|}(-x+x_n)}{(1+e^{-2(p_n-1)\sqrt{|\mu_{2n}|}x})^{2}}dx\nonumber\\
=&-\frac{1+O(p_n-2)}{16}e^{-(2p_n-2)\sqrt{|\mu_{2n}|}x_n}\big[1-(4-2p_n)\sqrt{|\mu_{2n}|}x_n\big]\\
&\cdot \frac{1}{2(p_n-1)\sqrt{|\mu_{2n}|}}\Big[\frac{1}{1+e^{2(p_n-1)\sqrt{|\mu_{2n}|}x}}+\text{ln}(1+e^{2(p_n-1)\sqrt{|\mu_{2n}|}x})\Big]\Big|_{-x_n/2}^{x_n/2}\nonumber\\
&-\frac{1}{64}(2-p_n)\sqrt{|\mu_{2n}|}x_n^2e^{-(2p_n-2)\sqrt{|\mu_{2n}|}x_n}+o(e^{-2\sqrt{|\mu_{2n}|}x_n})\nonumber\\
=&-\frac{1}{16}\big(\frac{x_n}{2}-2\big)e^{-2\sqrt{|\mu_{2n}|}x_n}-\frac{1+o(1)}{64}(2-p_n)\sqrt{|\mu_{2n}|}x_n^2e^{-2\sqrt{|\mu_{2n}|}x_n}+o(e^{-2\sqrt{|\mu_{2n}|}x_n})
\nonumber
\end{align}
as $n\to\infty$, and
\begin{align}\label{a.17}
&(p_n-1)\int^{x_n/2}_{-3x_n/4} w_{1n}'w_{1n}^{2p_n-3}w_{2n}^2dx\nonumber\\
=&(p_n-1)\int^{x_n/2}_{-3x_n/4} w_{1n}'w^{2p_n-3}_{1n}w_{2n}^{2p_n-2}\big[1+(4-2p_n)\text{ln}w_{2n}\big]dx\nonumber\\
=&\frac{1+O(p_n-2)}{16}e^{-(2p_n-2)\sqrt{|\mu_{2n}|}x_n}\nonumber\\
&\cdot\int^{x_n/2}_{-3x_n/4}\frac{e^{-2(p_n-1)\sqrt{|\mu_{2n}|}x}-1}{(1+e^{-2(p_n-1)
\sqrt{|\mu_{2n}|}x})^{3}}\Big[1-(4-2p_n)\sqrt{|\mu_{2n}|}(x_n-x)\Big]dx\\
=&-\frac{1+O(p_n-2)}{16}\frac{e^{-(2p_n-2)\sqrt{|\mu_{2n}|}x_n}}{2(p_n-1)\sqrt{|\mu_{2n}|}}\big[1-(4-2p_n)\sqrt{|\mu_{2n}|}x_n\big]\nonumber\\
&\cdot\Big[\frac{3e^{2(p_n-1)\sqrt{|\mu_{2n}|}x}+2}{(1+e^{2(p_n-1)\sqrt{|\mu_{2n}|}x})^2}+\text{ln}(1+e^{2(p_n-1)\sqrt{|\mu_{2n}|}x})\Big]\Big|^{x_n/2}_{-3x_n/4}\nonumber\\
&-\frac{1}{64}(2-p_n)\sqrt{|\mu_{2n}|}x_n^2e^{-(2p_n-2)\sqrt{|\mu_{2n}|}x_n}+o(e^{-2\sqrt{|\mu_{2n}|}x_n})\nonumber\\
=&-\frac{1}{16}\big(\frac{x_n}{2}-4\big)e^{-2\sqrt{|\mu_{2n}|}x_n}-\frac{1+o(1)}{64}(2-p_n)\sqrt{|\mu_{2n}|}x_n^2e^{-2\sqrt{|\mu_{2n}|}x_n}+o(e^{-2\sqrt{|\mu_{2n}|}x_n})
\nonumber
\end{align}
as $n\to\infty$. We hence obtain from above that \eqref{b.1} holds true. Lemma \ref{lemA.1} is thus proved. \qed

\begin{lem}\label{lemA.2}
Suppose $\lim\limits_{p_n\nearrow2}(2-p_n)x_n^2=48$, where $x_n>0$ is as in \eqref{3.16a}. Then we have
\begin{equation}\label{a.13}
\begin{split}
&H_n:=\int_{x_n/2}^{3x_n/2}xw_{1n}w_{1n}'w_{2n}^{2p_n-2}dx
+(p_n-1)\int_{-3x_n/4}^{x_n/2}xw^{2p_n-3}_{1n}w_{1n}'w_{2n}^2dx\\
&\ \ \ \ \ =\big[-\frac{x_n^2}{32}-\frac{x_n}{16}\big]e^{-2\sqrt{|\mu_{2n}|}x_n}+o(x_ne^{-2\sqrt{|\mu_{2n}|}x_n})\ \   \text{as}\ \ p_n\nearrow2,
\end{split}
\end{equation}
and
\begin{equation}\label{a.29}
\begin{split}
&a_n\sqrt{1-b_n^2}(\mu_{1n}-\mu_{2n})\inte xw_{1n}'w_{2n}dx\\
=&\frac{1}{64}x_n^2e^{-\sqrt{|\mu_{2n}|}x_n}+o(x_ne^{-2\sqrt{|\mu_{2n}|}x_n})\ \ \text{as}\ \ p_n\nearrow2.
\end{split}
\end{equation}
\end{lem}

\noindent \textbf{Proof.}  Since $\lim\limits_{p_n\nearrow2}(2-p_n)x_n^2=48$,
 \begin{align*}
 	&\int^{x_n/2}_{-x_n/2}\frac{x^2}{(1+e^{-2(p_n-1)\sqrt{|\mu_{2n}|}x})^{2}}dx
 	=\big[1+o(1)\big]\frac{x_n^3}{24}\ \ \ \text{as}\ \ p_n\nearrow2,
\end{align*}
and
\begin{align*}
 	&\int^{x_n/2}_{-3x_n/4}x^2\frac{1-e^{-2(p_n-1)\sqrt{|\mu_{2n}|}x}}{(1+e^{-2(p_n-1)\sqrt{|\mu_{2n}|}x})^{3}}dx
 	=\big[1+o(1)\big]\frac{x_n^3}{24}\ \ \ \text{as}\ \ p_n\nearrow2,
\end{align*}
we deduce from \eqref{a.13a}--\eqref{a.16a} that
 \begin{align*}
 	&\int_{x_n/2}^{3x_n/2}xw_{1n}w_{1n}'w_{2n}^{2p_n-2}dx\\
 =&-\frac{1+O(p_n-2)}{16}\int^{3x_n/2}_{x_n/2}\frac{xe^{-2\sqrt{|\mu_{2n}|}x}e^{2(p_n-1)
 \sqrt{|\mu_{2n}|}(x-x_n)}}{(1+e^{2(p_n-1)\sqrt{|\mu_{2n}|}(x-x_n)})^2}dx\\
 	=&-\frac{1+O(p_n-2)}{16}e^{-(2p_n-2)\sqrt{|\mu_{2n}|}x_n}  	 \int^{x_n/2}_{-x_n/2}\big(-x+x_n\big)\frac{1-(4-2p_n)\sqrt{|\mu_{2n}|}(-x+x_n)}{(1+e^{-2(p_n-1)\sqrt{|\mu_{2n}|}x})^{2}}dx\\
 	=&-\frac{3x_n^2}{128}e^{-2\sqrt{|\mu_{2n}|}x_n}
 	+o\big(x_ne^{-2\sqrt{|\mu_{2n}|}x_n}\big)\ \ \ \text{as}\ \ p_n\nearrow2,
 \end{align*}
 and
 \begin{align*}
 	&(p_n-1)\int^{x_n/2}_{-3x_n/4} xw_{1n}'w_{1n}^{2p_n-3}w_{2n}^2dx\\
 	=&(p_n-1)\int^{x_n/2}_{-3x_n/4} xw_{1n}'w^{2p_n-3}_{1n}w_{2n}^{2p_n-2}\big[1+(4-2p_n)\text{ln}w_{2n}\big]dx\\
 	=&\frac{1+O(p_n-2)}{16}e^{-(2p_n-2)\sqrt{|\mu_{2n}|}x_n}\\ 	 &\cdot\int^{x_n/2}_{-3x_n/4}x\frac{e^{-2(p_n-1)\sqrt{|\mu_{2n}|}x}-1}{(1+e^{-2(p_n-1)\sqrt{|\mu_{2n}|}x})^{3}}\big[1-(4-2p_n)\sqrt{|\mu_{2n}|}(x_n-x)\big]dx\\
 	=&\Big(-\frac{x_n^2}{128}-\frac{x_n}{16}\Big)e^{-2\sqrt{|\mu_{2n}|}x_n}+o\big(x_ne^{-2\sqrt{|\mu_{2n}|}x_n}\big)\ \ \ \text{as}\ \ p_n\nearrow2,
 \end{align*}
 which thus proves \eqref{a.13}.

 Moreover,  since
 \begin{align*}
 	\inte xw_{1n}'w_{2n}dx=&-\inte w_{1n}w_{2n}dx-\inte xw_{1n}w'_{2n}dx\\
 	=&-\inte w_{1n}w_{2n}dx+\inte (-x+x_n)w'_{1n}w_{2n}dx+o(x_ne^{-\sqrt{|\mu_{2n}|}x_n})
 \end{align*}
 as  $p_n\nearrow2$, we derive from \eqref{b.12} that
 \begin{equation*}\label{a.10}
 	\begin{split}&a_n\sqrt{1-b_n^2}(\mu_{1n}-\mu_{2n})\inte xw_{1n}'w_{2n}dx\nonumber\\
 		=&\frac{x_n}{2}a_n\sqrt{1-b_n^2}(\mu_{1n}-\mu_{2n})\inte w_{1n}'w_{2n}dx+\frac{1+o(1)}{16}x_ne^{-2\sqrt{|\mu_{2n}|}x_n}\\
 		=&\frac{1}{64}x_n^2e^{-\sqrt{|\mu_{2n}|}x_n}+o(x_ne^{-2\sqrt{|\mu_{2n}|}x_n})\ \ \ \text{as}\ \ p_n\nearrow2,\nonumber
 \end{split}\end{equation*}
 which thus yields that \eqref{a.29} holds true. This therefore proves Lemma \ref{lemA.2}.
 \qed


\begin{thebibliography}{99}


\bibitem{con1} T. Cazenave,  Semilinear Schr\"odinger Equations, Courant Lecture Notes in Mathematics  Vol. 10, Courant Institute of
Mathematical Science/AMS, New York (2003).



		
\bibitem{ii} R. L. Frank, D. Gontier and M. Lewin, {\em The nonlinear Schr\"{o}dinger equation for orthonormal functions II: application to Lieb-Thirring inequalities}, Comm. Math. Phys. {\bf 384} (2021), 1783--1828.
		

		
\bibitem{elli} D. Gilbarg and N. S. Trudinger, Elliptic Partial Differential Equations, 2nd ed., Belin: Springer, (1997).

		
\bibitem{i} D. Gontier, M. Lewin and F. Q. Nazar, \emph{The nonlinear Schr$\ddot{o}$dinger equation for orthonormal functions: existence of ground states}, Arch. Ration. Mech. Anal. {\bf240} (2021), 1203--1254.

\bibitem{G}
X. W. Guan, M. T. Batchelor, C. H. Lee, \emph{Fermi gases in one dimension: From Bethe Ansatz to experiments}, Rev. Mod. Phys. {\bf 85} (2013), 1633--1691.

\bibitem{glw}  Y. J. Guo, C. S. Lin and J. C. Wei, {\em Local uniqueness and refined spike profiles of ground states for two-dimensional attractive Bose-Einstein condensates}, SIAM J. Math. Anal. {\bf 49} (2017), 3671--3715.

\bibitem{GLY} Y. J. Guo, Y. Luo and W. Yang, {\em The nonexistence of vortices for rotating Bose-Einstein condenstates with attractive interactions},  Arch. Ration. Mech. Anal. {\bf238} (2020), 1231--1281.

\bibitem{GS} Y. J. Guo and R. Seiringer, {\em On the mass concentration for Bose-Einstein condensates with attractive interactions}, Lett. Math. Phys. {\bf104} (2014), 141--156.

		
%
%
%
%
%
%

\bibitem{K} P. Kevrekidis and D. Frantzeskakis, \emph{Solitons in coupled nonlinear Schr\"{o}dinger models: a survey of recent developments}, Rev. Phys. {\bf 1} (2016), 140--153.


\bibitem{uniq} M. K. Kwong,  \emph{Uniqueness of positive solutions of $\Delta u - u + u^p = 0$ in $\R^N$}, Arch. Ration. Mech. Anal. {\bf105} (1989), 243--266.

\bibitem{Lewin} M. Lewin,  \emph{Geometric methods for nonlinear many-body quantum systems}, J. Funct. Anal. {\bf 260} (2011), 3535--3595.



%
%

\bibitem{1981} E. H. Lieb, \emph{Variational principle for many-fermion systems}, Phys. Rev. Lett. {\bf 46} (1981), 457--459
		
\bibitem{analysis} E. H. Lieb and M. Loss,   Analysis, Graduate Studies in Mathematics Vol. 14, 2nd ed, American Mathematical Society, Providence, RI, 2001.
%
%
%
%
\bibitem{concen} P. L. Lions,  \emph{The concentration-compactness principle in the calculus of variations. The locally compact case. I}, Ann. Inst. Henri Poincar\'{e}. Anal. Non Lin\'{e}aire {\bf1} (1984), 109--145.

%
\bibitem{concen2} P. L. Lions,  \emph{The concentration-compactness principle in the calculus of variations. The locally compact case. II}, Ann. Inst H. Poincar\'{e}. Anal. Non Lineaire  {\bf1} (1984), 223--283.

\bibitem{Wei} J. C. Wei and  M. Winter, {\em Existence, classification and stability analysis of multiple-peaked solutions for the Gierer-Meinhardt system in  $\mathbb{R}^1$}, Methods Appl. Anal. {\bf 14} (2007),   119--163.

%
%
%
%
%
%
%
%
%


\end{thebibliography}
\end{document}